\documentclass{tran-l}

\newtheorem{theorem}{Theorem}[section]
\newtheorem{lemma}[theorem]{Lemma}
\newtheorem{proposition}[theorem]{Proposition}
\newtheorem{corollary}[theorem]{Corollary}

\theoremstyle{definition}

\theoremstyle{remark}
\newtheorem{remark}[theorem]{Remark}

\numberwithin{equation}{section}

\newenvironment{axiom}{\begin{list}{$\bullet$}{\setlength{\labelsep}{.3cm}%
\setlength{\leftmargin}{2cm}\setlength{\rightmargin}{0cm}%
\setlength{\labelwidth}{1.8cm}\setlength{\itemsep}{0pt}}}{\end{list}}
\newcommand{\ax}[1]{\item[{\bf #1}\hfill]\index{#1}}

\newcommand{\QED}{{\unskip\nobreak\hfil\penalty50%
\hskip1em\hbox{}\nobreak\hfil $\Box$%
\parfillskip=0pt \finalhyphendemerits=0 \par\medskip\noindent}}

\newcommand{\bfind}[1]{{\bf #1}}
\newcommand{\n}{\par\noindent}
\newcommand{\nn}{\par\vskip2pt\noindent}
\newcommand{\sn}{\par\smallskip\noindent}
\newcommand{\mn}{\par\medskip\noindent}
\newcommand{\bn}{\par\bigskip\noindent}
\newcommand{\pars}{\par\smallskip}
\newcommand{\parm}{\par\medskip}
\newcommand{\parb}{\par\bigskip}

\newcommand{\ovl}[1]{\overline{#1}}
\newcommand{\ec}{\prec_{\exists}}
\newcommand{\sep}{^{\rm sep}}

\newcommand{\chara}{\mbox{\rm char}\,}
\newcommand{\trdeg}{\mbox{\rm trdeg}\,}
\newcommand{\rr}{\mbox{\rm rr}\,}
\newcommand{\N}{\mathbb{N}}
\newcommand{\Q}{\mathbb{Q}}

\newcommand{\Z}{\mathbb{Z}}
\newcommand{\F}{\mathbb{F}}
\newcommand{\Fp}{\F_p}

\begin{document}

\title[Stability Theorem]{Elimination of Ramification I:
The Generalized Stability Theorem}

\author{Franz-Viktor Kuhlmann}
\address{Department of Mathematics and Statistics,
   University of Saskatchewan,
   106 Wiggins Road,
   Saskatoon, Saskatchewan, Canada S7N 5E6}
\email{fvk@math.usask.ca}
\thanks{I wish to thank Peter Roquette for his invaluable help and
support. I also thank F.~Delon, B.~Green, A.~Prestel, F.~Pop, and
R.~Strebel for inspiring discussions and suggestions that helped me
prove the Generalized Stability Theorem. During the work on this paper,
I was partially supported by a Canadian NSERC grant. I have also been a
member of the Newton Institute in Cambridge and a guest at the
University of Konstanz; I gratefully acknowledge their support.}


\date{July 11, 2009}

\begin{abstract}
We prove a general version of the ``Stability Theorem'': if $K$ is a
valued field such that the ramification theoretical defect is trivial
for all of its finite extensions, and if $F|K$ is a finitely generated
(transcendental) extension of valued fields for which equality holds in
the Abhyankar inequality, then the defect is also trivial for all
finite extensions of $F$. This theorem is applied to eliminate
ramification in such valued function fields. It has applications to
local uniformization and to the model theory of valued fields in
positive characteristic.
\end{abstract}

\maketitle

%
%ÄÄÄÄÄÄÄÄÄÄÄÄÄÄÄÄÄÄÄÄÄÄÄÄÄÄÄÄÄÄÄÄÄÄÄÄÄÄÄÄÄÄÄÄÄÄÄÄÄÄÄÄÄÄÄÄÄÄÄÄÄÄÄÄÄÄÄÄÄÄÄ
%
\section{Introduction}
%
%
%ÄÄÄÄÄÄÄÄÄÄÄÄÄÄÄÄÄÄÄÄÄÄÄÄÄÄÄÄÄÄÄÄÄÄÄÄÄÄÄÄÄÄÄÄÄÄÄÄÄÄÄÄÄÄÄÄÄÄÄÄÄÄÄÄÄÄÄÄÄÄÄ
%
\subsection{The Main Theorem}
In this paper, we consider the \bfind{defect} (also called
\bfind{ramification deficiency}) of finite extensions of valued fields.
For a valued field $(K,v)$, we will denote its value group by $vK$ and
its residue field by $\ovl{K}$ or by $Kv$. An extension of valued fields
is written as $(L|K,v)$, meaning that $v$ is a valuation on $L$ and $K$
is equipped with the restriction of this valuation. Every finite
extension $L$ of a valued field $(K,v)$ satisfies the \bfind{fundamental
inequality} (cf.\ [En], [Z--S]):
\begin{equation}                             \label{fundineq}
n\>\geq\>\sum_{i=1}^{\rm g} {\rm e}_i {\rm f}_i
\end{equation}
where $n=[L:K]$ is the degree
of the extension, $v_1,\ldots,v_{\rm g}$ are the distinct extensions
of $v$ from $K$ to $L$, ${\rm e}_i=(v_iL:vK)$ are the respective
ramification indices and ${\rm f}_i=[Lv_i:Kv]$ are the respective
inertia degrees. Note that ${\rm g}=1$ if $(K,v)$ is henselian.

We say that $(K,v)$ is \bfind{defectless in} $L$ if equality holds in
(\ref{fundineq}). Further, $(K,v)$ is called a \bfind{defectless} (or
\bfind{stable}) field if it is defectless in every finite extension $L$
of $K$. Finally, $(K,v)$ is called a \bfind{separably defectless} field
if it is defectless in every finite separable extension, and an
\bfind{inseparably defectless} field if it is defectless in every finite
purely inseparable extension of $K$. Note that every trivially valued
field is a defectless field.

Every valued field of residue characteristic $0$ is defectless; this is
a consequence of the ``Lemma of Ostrowski'' (see Section~\ref{sectdef}).
So we will always assume in the following that
\[p=\mbox{\rm char}(\ovl{K})>0\;.\]

Let $(L|K,v)$ be an extension of valued fields of finite
transcendence degree. Then the following well known form of
the ``Abhyankar inequality'' holds (it is a consequence of
Lemma~\ref{prelBour} below):
\begin{equation}                            \label{wtdgeq}
\trdeg L|K \>\geq\> \rr vL/vK \,+\, \trdeg \ovl{L}|\ovl{K}\;,
\end{equation}
where $\rr vL/vK:=\dim_{\Q}\,(vL/vK)\otimes\Q$ is the
\bfind{rational rank} of the abelian group $vL/vK$, i.e., the maximal
number of rationally independent elements in $vL/vK$. We will say that
$(L|K,v)$ is \bfind{without transcendence defect} if equality holds in
(\ref{wtdgeq}).

In this paper, \bfind{function field} will always mean {\it algebraic
function field}, where we include the case of finite algebraic
extensions in order to avoid case distinctions. The paper is devoted to
the proof of the following theorem, which has important applications.

\begin{theorem}                                    \label{ai}
{\bf (Generalized Stability Theorem)}\n
Let $(F|K,v)$ be a valued function field without transcendence
defect. If $(K,v)$ is a defectless field, then $(F,v)$ is a defectless
field. The same holds for ``inseparably defectless'' in the place of
``defectless''. If $vK$ is cofinal in $vF$, then it also holds for
``separably defectless'' in the place of ``defectless''.
\end{theorem}

Take a valued field $(K,v)$ and fix an extension of the valuation $v$ to
the separable-algebraic closure $K\sep$ of $K$. Then the \bfind{absolute
ramification field} $K^r$, the \bfind{absolute inertia field} (also
called \bfind{strict henselization}) $K^i$, and the
\bfind{henselization} $K^h$ of $(K,v)$ (with respect to the chosen
extension of $v$) are the ramification field, the inertia field and the
decomposition field, respectively, of the extension $(K\sep|K,v)$. Since
all extensions of the valuation $v$ from $K$ to $K\sep$ are conjugate
(i.e., are obtained from each other by composing with an automorphism of
$K\sep|K$), these fields are unique up to valuation preserving
isomorphism. Therefore, we will often speak of ``{\it the}
henselization'' and work with $K^h$ without previously fixing an
extension of the valuation. If $K^h=K$, then $(K,v)$ is called
\bfind{henselian}. This holds if and only if the extension of $v$ from
$K$ to every algebraic extension field is unique (implying that g$=1$ in
(\ref{fundineq})), or equivalently, if and only if $(K,v)$ satisfies
Hensel's Lemma. An algebraic extension of a henselian field $(K,v)$ is
called \bfind{purely wild} if it is linearly disjoint from $K^r|K$.

The algebraic closure of a field $K$ will be denoted by $\tilde{K}$. If
$F|K$ is an arbitrary field extension and $L|K$ is an algebraic
extension, then $L.F$ will denote the field compositum of $L$ and $F$
inside of $\tilde{F}$.

\begin{corollary}                                \label{aife}
Let $(F|K,v)$ be a valued function field without transcendence
defect, and $E|F$ a finite extension. Fix an extension of $v$ from $F$
to $\tilde{K}.F\,$. Then there is a finite extension $L_0|K$ such that
for every algebraic extension $L$ of $K$ containing $L_0\,$, $(L.F,v)$
is defectless in $L.E\,$. If $(K,v)$ is henselian, then $L_0|K$ can be
chosen to be purely wild.
\end{corollary}

Theorem~\ref{ai} was stated and proved in [K1]; the proof presented here
is an improved version. The theorem is a generalization of a result of
Grauert and Remmert [G--R] which is restricted to the case of
algebraically closed complete ground fields of rank 1 (i.e., with
archimedean ordered value group). A generalization of the
Grauert--Remmert Theorem was given by Gruson [G]. A good presentation
of it can be found in the book [B--G--R] of Bosch, G\"untzer and Remmert
(\S 5.3.2, Theorem 1). The proof uses methods of nonarchimedean
analysis. Further generalizations are due to Matignon and Ohm;
see also [G--M--P]. In [O], Ohm arrived independently of [K1]
at a quite general version of the Stability Theorem. But like all of its
forerunners, it is still restricted to the case $\mbox{\rm trdeg}(F|K) =
\mbox{\rm trdeg}(\ovl{F}|\ovl{K})$ and is therefore not sufficient for
the applications we will list below. Ohm deduces his theorem from
[B--G--R], Proposition 3, p.\ 215 (more precisely, from a generalized
version of this proposition which is proved but not stated in
[B--G--R]).

In contrast to this approach, we give a new proof which replaces the
analytic methods of [B--G--R] by valuation theoretical arguments. Such
arguments seem to be more adequate for a theorem that is of (Krull)
valuation theoretical nature. First, we reduce the proof to the study of
Galois extensions of degree $p$ of special henselized function
fields $F$. We deduce normal forms which allow us to read off that the
extension is defectless. Related results can be found in the work of
Hasse, Whaples, Epp [E], and in Matignon's proof of his version of
Theorem~\ref{ai}.

%
%ÄÄÄÄÄÄÄÄÄÄÄÄÄÄÄÄÄÄÄÄÄÄÄÄÄÄÄÄÄÄÄÄÄÄÄÄÄÄÄÄÄÄÄÄÄÄÄÄÄÄÄÄÄÄÄÄÄÄÄÄÄÄÄÄÄÄÄÄÄÄÄ
%
\subsection{Applications}
\mbox{ }\n
$\bullet$ \ {\bf Elimination of ramification.} \
This is the task of finding a (nice) transcendence basis $\mathcal{T}$
for a given valued function field $(F|K,v)$ such that for some
extension of $v$ to the algebraic closure of $F$, the extension $(F^h|
K(\mathcal{T})^h ,v)$ of respective henselizations is \bfind{unramified},
that is, the residue fields form a separable extension $\ovl{F}|
\ovl{K(\mathcal{T})}$ of degree equal to $[F^h:K(\mathcal{T})^h]$, and
$vF=vK(\mathcal{T})$. (Note that passing to the henselization
does not change value group and residue field.) In [K-K1] we use
Theorem~\ref{ai} to prove:

\begin{theorem}                             \label{hrwtd}
Take a defectless field $(K,v)$ and a valued function field
$(F|K,v)$
without transcendence defect. Assume that $\ovl{F}|\ovl{K}$ is a
separable extension and $vF/vK$ is torsion free. Then $(F|K,v)$ admits
elimination of ramification in the following sense: there is a standard
valuation transcendence basis $\mathcal{T}=\{x_1,\ldots,x_r,y_1,\ldots,
y_s\}$ of $(F|K,v)$ such that
\pars
a) \ $vF= vK\oplus\Z vx_1\oplus \ldots\oplus
\Z vx_r$,\par
b) \ $\ovl{y}_1,\ldots,\ovl{y}_s$ form a separating transcendence
basis of $\ovl{F}|\ovl{K}$.
\sn
For each such transcendence basis $\mathcal{T}$ and every extension of $v$
to the algebraic closure of $F$, $(F^h|K(\mathcal{T})^h,v)$ is unramified.
\end{theorem}

A \bfind{standard valuation transcendence basis} of an extension
$(F|K,v)$ of valued fields is a transcendence basis $\mathcal{T}=
\{x_i,y_j\mid i\in I,\,j\in J\}$ where the values $vx_i\,$, $i\in I$,
are rationally independent values in $vF$ modulo $vK$, and the
residues $\ovl{y}_j$, $j\in J$, are algebraically independent over
$\ovl{K}$. In this paper, we deduce:

\begin{corollary}                             \label{elram}
Let $(F|K,v)$ be a valued function field without transcendence
defect. Fix an extension of $v$ to $\tilde{F}$. Then there is a finite
extension $L_0|K$ and a standard valuation transcendence basis
$\mathcal{T}$ of $(L_0.F|L_0,v)$ such that for every algebraic extension
$L$ of $K$ containing $L_0\,$, the extension $((L.F)^h|L(\mathcal{T})^h,v)$
is unramified.
\end{corollary}

\mn
$\bullet$ \ {\bf Local uniformization in positive and in mixed
characteristic.} \ Theorem~\ref{hrwtd} is a crucial ingredient for our
proof that Abhyankar places of function fields (i.e., places
for which the corresponding valued function field is without
transcendence defect) admit local uniformization (cf.\ [K--K1], [K2],
[K3]), provided that a (necessary) separability
condition for their residue field extension is satisfied. Together with
a theorem proven in [K8], it is also the main ingredient for local
uniformization of arbitrary places after a finite extension of the
function field ([K--K2]). The analogous arithmetic cases (also treated
in [K--K1] and [K--K2]) use Theorem~\ref{ai} in mixed characteristic.
The proofs use solely valuation theory.

Let us mention at this point that our approach has some
similarity with S.~Abhyankar's method of using ramification theory
in order to reduce the question of resolution of singularities to the
study of Galois extensions of degree $p$ and the search for suitable
normal forms of Artin--Schreier--like minimal polynomials.

\mn
$\bullet$ \ {\bf Model theory of valued fields.} \ If $\mathcal{L}$
is an elementary language and $\mathcal{A}\subset \mathcal{B}$ are
$\mathcal{L}$-structures, then we will say that $\mathcal{A}$ is
\bfind{existentially closed in} $\mathcal{B}$ and write
$\mathcal{A}\ec \mathcal{B}$ if every
existential sentence with parameters from $\mathcal{A}$ that holds in
$\mathcal{B}$ also holds in $\mathcal{A}$. When we talk of fields, then we use
the language of rings or fields. When we talk of valued fields, we
augment this language by a unary predicate for the valuation ring or a
binary predicate for valuation divisibility. For ordered
abelian groups, we use the language of groups augmented by a binary
predicate for the ordering. For the meaning of ``existentially closed
in'' in the setting of valued fields and of ordered abelian groups, see
[K--P]. In [K7] we use Theorem~\ref{ai} to prove the following
Ax--Kochen--Ershov Principle:
\begin{theorem}                             \label{AKEwtd}
Take a henselian defectless valued field $(K,v)$ and an extension
$(L|K,v)$ of finite transcendence degree without transcendence defect.
If $vK$ is existentially closed in $vL$ and $\ovl{K}$ is existentially
closed in $\ovl{L}$, then $(K,v)$ is existentially closed in $(L,v)$.
\end{theorem}
\n
This in turn is used to prove Ax--Kochen--Ershov Principles and
further model theoretic results for tame valued fields.

%
%ÄÄÄÄÄÄÄÄÄÄÄÄÄÄÄÄÄÄÄÄÄÄÄÄÄÄÄÄÄÄÄÄÄÄÄÄÄÄÄÄÄÄÄÄÄÄÄÄÄÄÄÄÄÄÄÄÄÄÄÄÄÄÄÄÄÄÄÄÄÄÄ
%
\section{Preliminaries}
%
%
%ÄÄÄÄÄÄÄÄÄÄÄÄÄÄÄÄÄÄÄÄÄÄÄÄÄÄÄÄÄÄÄÄÄÄÄÄÄÄÄÄÄÄÄÄÄÄÄÄÄÄÄÄÄÄÄÄÄÄÄÄÄÄÄÄÄÄÄÄÄÄÄ
%
\subsection{Facts from general valuation theory}
We will use the following facts throughout the paper, often without
citing. The first lemma is a consequence of the fundamental inequality
(\ref{fundineq}).
\begin{lemma}                           \label{aaa}
Let $(L|K,v)$ be an algebraic extension of valued fields. Then
$vL/vK$ is a torsion group and the extension $Lv|Kv$ of residue fields
is algebraic. If $L|K$ is finite, then so are $vL/vK$ and $Lv|Kv$. If
$v$ is trivial on $K$ (i.e., $vK=\{0\}$), then $v$ is trivial on~$L$.

For each extension of the valuation $v$ from $K$ to its algebraic
closure $\tilde{K}$, the value group of $\tilde{K}$ is divisible
and the residue field of $\tilde{K}$ is algebraically closed.
\end{lemma}

An extension $(L|K,v)$ is called \bfind{immediate} if the canonical
embeddings of $vK$ in $vL$ and of $Kv$ in $Lv$ are onto.
For the next results, and general background in valuation theory, we
refer the reader to [En], [R], [W], [Z--S].

\begin{lemma}                               \label{henseliz}
The henselization $K^h$ of a valued field $(K,v)$ (which is unique up to
valuation preserving isomorphism over $K$) is an immediate
separable-algebraic extension and has the following universal property:
if $(L,v')$ is an arbitrary henselian extension field of $(K,v)$, then
there is a unique valuation preserving embedding of $(K^h,v)$ in
$(L,v')$ over $K$. Thus, a henselization of $(K,v)$ can be chosen in
every henselian valued extension field of $(K,v)$.
\end{lemma}

\begin{lemma}                               \label{algehens}
An algebraic extension of a henselian valued field, equipped with the
unique extension of the valuation, is again henselian.

If $L|K$ is algebraic, then $(L.K^h,v)$ is the henselization of $(L,v)$
(with respect to the unique extension of $v$ from $K^h$ to $\tilde{K}$).
\end{lemma}

Take any valued field $(K,w)$ and a valuation $\ovl{w}$ on the residue
field $Kw$. Then the \bfind{composition} $w\circ\ovl{w}$ is the
valuation whose valuation ring is the subring of the valuation ring of
$w$ consisting of all elements whose $w$-residues lie in the valuation
ring of $\ovl{w}$. (Note that we identify equivalent valuations.) While
$w\circ\ovl{w}$ does actually not mean the composition of $w$ and
$\ovl{w}$ as mappings, this notation is used because, up to
equivalence the place associated with $w\circ\ovl{w}$ is indeed the
composition of the places associated with $w$ and $\ovl{w}$.

For a valued field $(K,v)$, every convex subgroup $\Gamma$ of $vK$ gives
rise to a valuation $v_\Gamma$ such that $v_\Gamma K$ is isomorphic to
$vK/\Gamma$ and the valuation ring of $v$ is contained in that of
$v_\Gamma\,$. Then $v$ induces a valuation $\ovl{v}_\Gamma$ on
$Kv_\Gamma\,$ (whose valuation ring is the image of the
valuation ring of $v$ under the place associated with $v_\Gamma$), such
that $v=v_\Gamma\circ\ovl{v}_\Gamma\,$. The value group $\ovl{v}_\Gamma
(Kv_\Gamma)$ of $\ovl{v}_\Gamma$ can be identified with $\Gamma$ and its
residue field $(Kv_\Gamma) \ovl{v}_\Gamma$ with $Kv$.

The \bfind{rank} of a valued field $(K,v)$ is the order type of the
chain of non-trivial convex subgroups of its value group $vK$. It has
rank $1$ (i.e., its only convex subgroups are $\{0\}$ and $vK$) if and
only if $vK$ is \bfind{archimedean}, that is, embeddable in the ordered
additive group of the reals. If $(K,v)$ has finite rank $n$, then $v$ is
the composition $v=v_1\circ\ldots\circ v_n$ of $n$ many rank $1$
valuations. The following lemma does in general not hold for valuations
of rank $>1$:

\begin{lemma}                               \label{rk1dense}
If $(K,v)$ is a valued field of rank $1$, then $K$ is dense in its
henselization. In particular, the completion of $(K,v)$ is henselian.
\end{lemma}

For the easy proof of the following lemma, see [B], chapter VI,
\S10.3, Theorem~1.
\begin{lemma}                                      \label{prelBour}
Let $(L|K,v)$ be an extension of valued fields. Take elements $x_i,y_j
\in L$, $i\in I$, $j\in J$, such that the values $vx_i\,$, $i\in I$,
are rationally independent over $vK$, and the residues $y_jv$, $j\in
J$, are algebraically independent over $Kv$. Then the elements
$x_i,y_j$, $i\in I$, $j\in J$, are algebraically independent over $K$.

Moreover, if we write
\[f\>=\> \displaystyle\sum_{k}^{} c_{k}\,
\prod_{i\in I}^{} x_i^{\mu_{k,i}} \prod_{j\in J}^{} y_j^{\nu_{k,j}}\in
K[x_i,y_j\mid i\in I,j\in J]\]
in such a way that for every $k\ne\ell$
there is some $i$ s.t.\ $\mu_{k,i}\ne\mu_{\ell,i}$ or some $j$ s.t.\
$\nu_{k,j}\ne\nu_{\ell,j}\,$, then
\begin{equation}                            \label{value}
vf\>=\>\min_k\, v\,c_k \prod_{i\in I}^{}
x_i^{\mu_{k,i}}\prod_{j\in J}^{} y_j^{\nu_{k,j}}\>=\>
\min_k\, vc_k\,+\,\sum_{i\in I}^{} \mu_{k,i} v x_i\;.
\end{equation}
That is, the value of the polynomial $f$ is equal to the least of the
values of its monomials. In particular, this implies:
\begin{eqnarray*}
vK(x_i,y_j\mid i\in I,j\in J) & = & vK\oplus\bigoplus_{i\in I}
\Z vx_i\\
K(x_i,y_j\mid i\in I,j\in J)v & = & Kv\,(y_jv\mid j\in J)\;.
\end{eqnarray*}
Moreover, the valuation $v$ on $K(x_i,y_j\mid i\in I,j\in J)$ is
uniquely determined by its restriction to $K$, the values $vx_i$ and
the residues $y_jv$.
\end{lemma}

\begin{corollary}                              \label{fingentb}
Take a valued function field $(F|K,v)$ without transcendence
defect and set $r=\rr vF/vK$ and $s=\trdeg\ovl{F}|\ovl{K}$. Choose
elements $x_1,\ldots,x_r$, $y_1,\ldots,y_s\in F$ such that the
values $vx_1,\ldots,vx_r$ are rationally independent over $vK$ and the
residues $\ovl{y}_1,\ldots,\ovl{y}_s$ are algebraically independent
over $\ovl{K}$. Then $\mathcal{T}=\{x_1,\ldots,x_r,y_1,\ldots,y_s\}$ is
a standard valuation transcendence basis $(F|K,v)$
Moreover, $vF/vK$ and the extension $\ovl{F}|\ovl{K}$ are
finitely generated.
\end{corollary}
\begin{proof}
By the previous theorem, the elements $x_1,\ldots,x_r,y_1,\ldots,
y_s$ are algebraically independent over $K$. Since $\trdeg F|K=r+s$ by
assumption, these elements form a basis and hence a standard valuation
transcendence basis of $(F|K,v)$.

It follows that the extension $F|K(\mathcal{T})$ is finite. By the
fundamental inequality (\ref{fundineq}), this yields that $vF/
vK(\mathcal{T})$ and $\ovl{F}| \ovl{K(\mathcal{T})}$ are finite. Since already
$vK(\mathcal{T})/vK$ and $\ovl{K(\mathcal{T})}|\ovl{K}$ are finitely generated
by the foregoing lemma, it follows that also $vF/vK$ and
$\ovl{F}|\ovl{K}$ are finitely generated.
\end{proof}

Since an infinite number of distinct convex subgroups in an ordered
abelian group give rise to an infinite number of rationally independent
elements in that group, the following is also a consequence of
Lemma~\ref{prelBour}:
\begin{corollary}                           \label{fatd=fr}
Every valued field of finite transcendence degree over its prime field
has finite rank.
\end{corollary}

\begin{lemma}                               \label{tdcomp}
If $(F|K,v)$ is a valued function field without transcendence
defect and $v=w\circ \ovl{w}$, then $(F|K,w)$ and $(Fw|Kw,\ovl{w})$ are
valued function fields without transcendence defect.
\end{lemma}
\begin{proof}
The value group $\ovl{w}(Fw)$ can be considered (in a canonical way) a
subgroup of $vF$ and then $wF$ is isomorphic to $vF/\ovl{w}(Fw)$; the
same holds for $K$ in the place of $F$. From this it follows that
$\rr vF/vK=\rr wF/wK+\rr\ovl{w}(Fw)/\ovl{w}(Kw)$. Hence,
\begin{eqnarray*}
\trdeg F|K & = & \rr vF/vK\,+\,\trdeg Fv|Kv\\
 & = & \rr wF/wK\,+\, \rr\ovl{w}(Fw)/\ovl{w}(Kw)\,+\,\trdeg
(Fw)\ovl{w}|(Kw)\ovl{w}\\
 & \leq & \rr wF/wK\,+\,\trdeg Fw|Kw\>\leq\>\trdeg F|K\;,
\end{eqnarray*}
and we find that equality must hold everywhere. The last equality then
shows that $(F|K,w)$ is without transcendence defect, which by
Corollary~\ref{fingentb} implies that $Fw|Kw$ is a function
field. We also obtain that
\[\rr \ovl{w}(Fw)/\ovl{w}(Kw)\,+\,\trdeg
(Fw)\ovl{w}|(Kw) \ovl{w}=\trdeg Fw|Kw\>,\]
which shows that $(Fw|Kw,\ovl{w})$ is without transcendence defect.
\end{proof}

%
%
%ÄÄÄÄÄÄÄÄÄÄÄÄÄÄÄÄÄÄÄÄÄÄÄÄÄÄÄÄÄÄÄÄÄÄÄÄÄÄÄÄÄÄÄÄÄÄÄÄÄÄÄÄÄÄÄÄÄÄÄÄÄÄÄÄÄÄÄÄÄÄÄ
%
\subsection{$p$-th roots of $1$-units}       \label{sectpthr}
A well known application of Hensel's Lemma shows that in every henselian
field, each $1$-unit (an element of the form $1+b$ with $vb>0$) is an
$n$-th power for every $n$ not divisible by the residue characteristic
$p$. If the latter is not the case, then one considers the ``level'' of
the $1$-unit. For our purposes, we need a more precise result for
a $1$-unit to have a $p$-th root in a henselian field of mixed
characteristic.

Throughout this paper, we will take $C$ to be an element in the
algebraic closure of $\Q$ such that $C^{p-1}=-p$, where $p>0$ is a
prime. Take a henselian field $(K,v)$ of characteristic 0 and residue
characteristic $p$. Extend the valuation $v$ to $\tilde K$. Note that
\[C^p = -pC\;\;\;\mbox{\ \ and\ \ }\;\;\; vC = \frac{1}{p-1} vp>0\;.\]
Consider the polynomial
\begin{equation}                            \label{pthroot}
X^p-(1+b)
\end{equation}
with $b\in K$. Performing the transformation
\begin{equation}                            \label{transf}
X\;=\;CY\,+\,1\;,
\end{equation}
dividing by $C^p$ and using that $C^p = -pC$, we obtain the polynomial
\begin{equation}                            \label{trpol}
f(Y)\;=\; Y^p \,+\, g(Y) \,-\, Y \,-\, \frac{b}{C^p}
\end{equation}
with
\begin{equation}                            \label{p-part}
g(Y) \;=\; \sum_{i=2}^{p-1} \left( \begin{array}{c}
p\\i\end{array}\right) C^{i-p} Y^i
\end{equation}
a polynomial with coefficients in $K(C)$ of value $>0$.

\begin{lemma}                                  \label{exC}
Take $(K,v)$ and $C$ as above. Then
$K$ contains $C$ if and only if it contains all $p$-th roots of unity.
\end{lemma}
\begin{proof}
Since $\chara\ovl{K} =p$, the restriction of $v$ to $\Q\subset K$
is the $p$-adic valuation. Since $(K,v)$ is henselian, it contains
$(\Q^h,v_p)$. Let $\eta\ne 1$ be a $p$-th root of unity. It suffices to
show that $\Q^h(\eta)=\Q^h(C)$. Applying the transformation
(\ref{transf}) to the polynomial (\ref{pthroot}) with $b=0$, we obtain
the polynomial $f(Y)= Y^p+ g(Y) - Y\in \Q(C)[Y]$ which splits over
$\Q^h$ by Hensel's Lemma because $\ovl{f}(Y)=Y^p-Y$ splits over $\Fp\,$.
Since the non-zero roots of $f$ have nonzero residue and thus value
zero, $v(\eta-1)= vC=vp/(p-1)$. We find that $(v\Q^h(\eta):v\Q^h) \geq
p-1$. Consequently,
\[[\Q^h(\eta):\Q^h]\>\geq\> p-1\>\geq\> [\Q^h(C):\Q^h]\>\geq\>
[\Q^h(\eta):\Q^h]\;,\]
showing that equality holds everywhere and that $\Q^h(\eta)=\Q^h(C)$.
\end{proof}

\begin{lemma}                                       \label{Sub}
Let $(K,v)$ be a henselian field containing all $p$-th roots of unity.
Then
\[vb\, > \frac{p}{p-1}vp \;\;\Rightarrow\;\; 1 + b\in
(K^{\times})^p\]
for all $b\in K$. In other words, every 1-unit of level $>
\frac{p}{p-1}vp$ has a $p$-th root in $K$.
\end{lemma}
\begin{proof}
Consider the polynomial (\ref{trpol}), derived from (\ref{pthroot}). If
$vb > \frac{p}{p-1}vp=vC^p$, then $\ovl{f}(Y) = Y^p - Y$, which splits
over $\ovl{K}$. By Hensel's Lemma, this implies that $f(Y)$ splits over
$K$. Via the transformation (\ref{transf}), it follows that $1+b$ has a
$p$-th root in $K$.
\end{proof}

\begin{corollary}               \label{1+y}
Let $(K,v)$ be a henselian field containing all $p$-th roots of unity.
Take any $1$-units $1+b$ and $1+c$ in $K$. Then:
\sn
a)\ \ $1 + b \in (1 + b + c)\cdot (K^{\times})^p\;\;$
if $\;vc > \frac{p}{p-1}vp\,$.\sn
b)\ \ $1 + b \in (1 + b + c)\cdot (K^{\times})^p\;\;$
if $\;1+c\in (K^{\times})^p$ and $vbc > \frac{p}{p-1}vp\,$.\sn
c)\ \ $1 + c^p + pc\in (K^{\times})^p\;\;$ if $\;vc^p > vp\,$.\sn
d)\ \ $1 + b -pc \in (1 + b + c^p)\cdot
(K^{\times})^p\;\;$ if $\;vb\geq \frac{1}{p-1}vp\;$ and $\;vc^p > vp\,$.
\end{corollary}
\begin{proof}
a):\ \ $1+b\in (1+b+c)(K^{\times})^p$ is true if the quotient
\[\frac{1+b+c}{1+b}\; =\; 1+\frac{c}{1+b}\]
is an element of $(K^{\times})^p$. By hypothesis we have $vb>0$ and
thus $v\frac{c}{1+b}=vc$. Now our assertion follows from
Lemma~\ref{Sub}.
\sn
b):\ \ An application of part a) shows that
\[\mbox{$(1 + b + c)\,\in\, (1 + b)(1 + c)\cdot (K^{\times})^p\;\;$
if $\;vbc > \frac{p}{p-1}vp$.}\]
The assertion of b) is an immediate consequence of this.
\sn
c):\ \ If $vc^p>vp$ then for every $i=2,\ldots,p-1$ we have
\[v \left( \begin{array}{c} p\\i\end{array}\right) c^i\>\geq\>
vp + 2vc\> >\> \frac{p+2}{p} vp\>\geq\>\frac{p}{p-1} vp\;;\]
note that the last inequality holds for every $p\geq 2$. This together
with assertion a) yields
\[1 + c^p + pc\,\in\, \left(1 + c^p + pc + \sum_{i=2}^{p-1}
\left( \begin{array}{c} p\\i\end{array}\right) c^i\right) (K^{\times})^p
\,=\, (1+c)^p (K^{\times})^p\,=\, (K^{\times})^p\;.\]
d):\ \ In view of part c), the assertion follows from part b) where $b$
is replaced by $b-pc$ and $c$ is replaced by $c^p +pc$. Note that b)
can be applied since $v(b-pc)(c^p+pc)>\frac{1}{p-1}vp + vp
=\frac{p}{p-1}vp$.
\end{proof}

%
%ÄÄÄÄÄÄÄÄÄÄÄÄÄÄÄÄÄÄÄÄÄÄÄÄÄÄÄÄÄÄÄÄÄÄÄÄÄÄÄÄÄÄÄÄÄÄÄÄÄÄÄÄÄÄÄÄÄÄÄÄÄÄÄÄÄÄÄÄÄÄÄ
%
\subsection{The defect}                     \label{sectdef}
Assume that $(L|K,v)$ is a finite extension and the extension of $v$
from $K$ to $L$ is unique. Then the Lemma of Ostrowski says that
\begin{equation}                            \label{LoO}
[L:K]\;=\; (vL:vK)\cdot [Lv:Kv]\cdot p^\nu \;\;\;\mbox{ with }\nu\geq 0
\end{equation}
where $p$ is the \bfind{characteristic exponent} of $Kv$, that is,
$p=\chara Kv$ if this is positive, and $p=1$ otherwise (cf.\ [En], [R]).
The factor
\[\mbox{\rm d}(L|K,v)\>:=\> p^\nu \>=\> \frac{[L:K]}{(vL:vK)[Lv:Kv]}\]
is called the \bfind{defect} of the extension $(L|K,v)$. If $\nu>0$,
then we talk of a \bfind{non-trivial} defect. If $[L:K]=p$ then
$(L|K,v)$ has non-trivial defect if and only if it is immediate, i.e.,
$(vL:vK)=1$ and $[Lv:Kv]=1$. If $\mbox{\rm d}(L|K,v)=1$, then we call
$(L|K,v)$ a \bfind{defectless extension}. Note that $(L|K,v)$ is always
defectless if $\chara Kv=0$. Therefore,
\begin{corollary}                               \label{lostr0}
Every valued field $(K,v)$ with $\chara Kv=0$ is a defectless field.
\end{corollary}

The following lemma shows that the defect is multiplicative.
This is a consequence of the multiplicativity of the degree of field
extensions and of ramification index and inertia degree. We leave the
straightforward proof to the reader.
\begin{lemma}                                       \label{md}
Fix an extension of a valuation $v$ from $K$ to its algebraic closure.
If $L|K$ and $M|L$ finite extensions and the extension of $v$ from $K$
to $M$ is unique, then
\begin{equation}         \label{pf}
\mbox{\rm d}(M|K,v) = \mbox{\rm d}(M|L,v)\cdot\mbox{\rm d}(L|K,v)
\end{equation}
In particular, $(M|K,v)$ is defectless if and only if $(M|L,v)$ and
$(L|K,v)$ are defectless.
\end{lemma}

\begin{theorem}                                    \label{dl-hdl}
Take a valued field $(K,v)$ and fix an extension of $v$ to $\tilde{K}$.
Then $(K,v)$ is defectless if and only if its henselization $(K,v)^h$ in
$(\tilde{K},v)$ is defectless. The same holds for ``separably
defectless'' and ``inseparably defectless'' in the place of
``defectless''.
\end{theorem}
\begin{proof}
For ``separably defectless'', our assertion follows directly from [En],
Theorem (18.2). The proof of that theorem can easily be adapted to prove
the assertion for ``inseparably defectless'' and ``defectless''. See
[K6] for more details.
\end{proof}

Since a henselian field has a unique extension of the valuation to every
algebraic extension field, we obtain:
\begin{corollary}                                \label{mdc1}
A valued field $(K,v)$ is defectless if and only if $\mbox{\rm d}
(L|K^h,v)=1$ for every finite extension $L|K^h$.
\end{corollary}

Using this corollary together with Lemma~\ref{md}, one easily shows:
\begin{corollary}                                \label{mdc2}
Every finite extension of a defectless field is again a
defectless field.
\end{corollary}

\begin{lemma}                             \label{defcow}
Let $(K,v)$ be a valued field with $v=w\circ\ovl{w}\,$. If $(K,w)$
and $(Kw,\ovl{w})$ are defectless fields, then so is $(K,v)$.
\end{lemma}
\begin{proof}
Take a finite extension $L|K$; we wish to show that equality holds in
(\ref{fundineq}). Let $w_1,\ldots,w_{{\rm g}_w}$ be all extensions of
$w$ from $K$ to $L$, and set e$^w_i=(w_iL:wK)$ and f$^w_i=[Lw_i:Kw]$ for
$1\leq i\leq {\rm g}_w\,$. Further, for $1\leq i\leq {\rm g}_w\,$, let
$\ovl{w}_{i1},\ldots,\ovl{w}_{i{\rm g}_i}$ be all extensions of
$\ovl{w}$
from $Kw$ to $Lw_i$, and set e$_{ij}=(\ovl{w}_{ij}(Lw_i):\ovl{w}(Kw))$
and f$_{ij}=[(Lw_i)\ovl{w}_{ij}:(Kw)\ovl{w}]=[L(w_i\circ\ovl{w}_{ij}):
Kv]$ for $1\leq j\leq {\rm g}_i\,$. Since $(K,w)$ is a defectless field,
we have $[L:K]=\sum_{i=1}^{{\rm g}_w} {\rm e}^w_i {\rm f}^w_i\,$. Since
$(Kw,\ovl{w})$ is a defectless field, we have ${\rm f}^w_i=[Lw_i:Kw]=
\sum_{j=1}^{{\rm g}_i} {\rm e}_{ij} {\rm f}_{ij}\,$. Using that
${\rm e}^w_i {\rm e}_{ij}=((w_i\circ\ovl{w}_{ij})L:vK)$, we obtain:
\[
[L:K]\>=\>\sum_{i=1}^{{\rm g}_w} {\rm e}^w_i \sum_{j=1}^{{\rm g}_i}
{\rm e}_{ij}{\rm f}_{ij}\>=\>\sum_{i=1}^{{\rm g}_w}\sum_{j=1}^{{\rm
g}_i} ((w_i\circ\ovl{w}_{ij})L:vK)[L(w_i\circ\ovl{w}_{ij}):Kv]\;.
\]
As the valuations $w_i\circ\ovl{w}_{ij}$, $1\leq i\leq {\rm g}_w\,$,
$1\leq j\leq {\rm g}_i\,$, are distinct extensions of $v$ from $K$ to
$L$, the fundamental inequality implies that they are in fact all
extensions, and we have proved that equality holds in (\ref{fundineq}).
\end{proof}

In [K6] we have proved the following:
\begin{proposition}                                  \label{dlta}
Let $(K,v)$ be a henselian field and $N$ an arbitrary algebraic
extension of $K$ within $K^r$. If $L|K$ is a finite extension, then
\[\mbox{\rm d}(L|K,v)\>=\> \mbox{\rm d}(L.N | N,v)\;.\]
Hence, $(K,v)$ is a defectless field if and only if $(N,v)$ is a
defectless field. The same holds for ``separably defectless'' and
``inseparably defectless'' in the place of ``defectless''.
\end{proposition}

%
%ÄÄÄÄÄÄÄÄÄÄÄÄÄÄÄÄÄÄÄÄÄÄÄÄÄÄÄÄÄÄÄÄÄÄÄÄÄÄÄÄÄÄÄÄÄÄÄÄÄÄÄÄÄÄÄÄÄÄÄÄÄÄÄÄÄÄÄÄÄÄÄ
%
\subsection{Valuation disjoint and valuation regular extensions}
Let $(L|K,v)$ be an extension of valued fields, and $\mathcal{B}\subset L$.
Then $\mathcal{B}$ is said to be a \bfind{valuation independent set}
{\bf in} $(L|K,v)$ if for every $n\in\N$ and each choice of distinct
elements $b_1,\ldots,b_n\in \mathcal{B}$ and arbitrary elements
$c_1,\ldots,c_n\in K$,
\[v\sum_{i=1}^{v} c_ib_i\>=\>\min_{1\leq i\leq n} vc_ib_i\;.\]
Further, $\mathcal{B}$ is called a \bfind{standard valuation independent
set} ({\bf in} $(L|K,v)$) if it is of the form $\mathcal{B}=\{b'_ib''_j\mid
1\leq i\leq k\,,\, 1\leq j\leq\ell\}$ where the values
$vb'_1,\ldots,vb'_k$ lie in distinct cosets of $vL$ modulo $vK$, and
$b''_1,\ldots,b''_\ell$ are elements of value $0$ whose residues are
$\ovl{K}$-linearly independent. Every standard valuation independent set
is a valuation independent set.

Let $(\Omega|K,v)$ be an extension of valued fields and $F|K$ and $L|K$
two subextensions of $\Omega|K$. We say that $(F|K,v)$ is
\bfind{valuation disjoint} {\bf from} $(L|K,v)$ ({\bf in} $(\Omega,v)$)
if every standard valuation independent set $\mathcal{B}$ of $(F|K,v)$ is
also a standard valuation independent set of $(L.F|L,v)$. (It is
possible that a standard valuation independent set $\mathcal{B}$ of
$(F|K,v)$ remains valuation independent over $(L,v)$ without remaining a
standard valuation independent set of $(L.F|L,v)$.)

Let $G$ and $G'$ be two subgroups of some group $\mathcal{G}$, and $H$ a
common subgroup of $G$ and $G'$.
We will say that the group extension $G|H$ is {\bf disjoint from} the
group extension $G'|H$ ({\bf in} $\mathcal{G}$) \index{disjoint group
extensions} if every two elements in $G$ that belong to distinct cosets
modulo $H$ will also belong to distinct cosets modulo $G'$. This holds
if and only if $G\cap G'=H$. Hence, our notion ``disjoint from'' is
symmetrical, like ``linearly disjoint from'' in the field case.

\begin{lemma}                               \label{charvaldis}
Let $(\Omega|K,v)$ be an extension of valued fields and $F|K$ and $L|K$
subextensions of $\Omega|K$. Then $(F|K,v)$ is valuation disjoint from
$(L|K,v)$ in $(\Omega,v)$ if and only if\sn
1)\ $vF|vK$ is disjoint from $vL|vK$ in $v\Omega$, and\nn
2)\ $\ovl{F}|\ovl{K}$ is linearly disjoint from $\ovl{L}|\ovl{K}$
in $\ovl{\Omega}$.\sn
Consequently, the notion of ``valuation disjoint'' is symmetrical.
\end{lemma}
\begin{proof}
Let $\mathcal{B}$ be a standard valuation independent set in $(F|K,v)$ of
the form as given in the definition. By condition 1), the values
$vb'_1,\ldots,vb'_k$ also lie in distinct cosets of $v(L.F)$ modulo
$vL$. By condition 2), the residues $b''_1,\ldots,vb''_\ell$ remain
$\ovl{L}$-linearly independent. Hence, $\mathcal{B}$ be a standard
valuation independent set in $(L.F|L,v)$.

For the converse, assume that condition 1) or 2) is not satisfied. If 1)
is not satisfied, then there are two elements $b,b'\in F$ whose values
belong to distinct cosets modulo $vK$ but to the same coset modulo
$vL$. Hence, $\{b,b'\}$ is a standard valuation independent set in
$(F|K,v)$, but not in $(L.F|L,v)$.
If 2) is not satisfied, then there are elements $b_1,\ldots,b_n\in F$
of value $0$ whose residues are $\ovl{K}$-linearly independent but not
$\ovl{L}$-linearly independent. Then $\{b_1,\ldots, b_n\}$ is a standard
valuation independent set in $(F|K,v)$, but not in $(L.F|L,v)$.
\end{proof}

Recall that an extension $F|K$ is called \bfind{regular} if it is
linearly disjoint from $\tilde{K}|K$, or equivalently, if it is
separable and $K$ is relatively algebraically closed in $F$. An
extension $(F|K,v)$ of valued fields will be called \bfind{valuation
regular} if it is valuation disjoint from $(\tilde{K}|K,v)$ in
$(\tilde{F},v)$ for some extension of $v$ from $F$ to $\tilde{F}$. Using
that $v\tilde{K}$ is the divisible hull of $K$ and $\ovl{\tilde{K}}$ is
the algebraic closure of $\ovl{K}$ (Lemma~\ref{aaa}), one deduces:

\begin{lemma}                               \label{charvalreg}
An extension $(F|K,v)$ is valuation regular if and only if\sn
1)\ $vF/vK$ is torsion free,\nn
2)\ $\ovl{F} | \ovl{K}$ is regular.\sn
Consequently, $(F|K,v)$ is valuation regular if and only if it is
valuation disjoint from $(\tilde{K}|K,v)$ for {\em every} extension
of $v$ from $F$ to $\tilde{F}$. Every valued field extension of an
algebraically closed valued field is valuation regular.
\end{lemma}

Since the henselization of a valued field is
an immediate extension, this lemma yields:
\begin{corollary}                           \label{valreg-h}
If $(F|K,v)$ is valuation regular, then also $(F^h|K^h,v)$ is valuation
regular.
\end{corollary}

Important examples of valuation regular extensions are the valued field
extensions which are generated by standard algebraically valuation
independent sets. Indeed, it follows from Lemma~\ref{prelBour} that they
satisfy the conditions of the above lemma. Using also
Lemma~\ref{charvaldis}, we obtain:
\begin{lemma}                                       \label{K(T)valreg}
Let $(\Omega|K,v)$ be an extension of valued fields containing a
standard algebraically valuation independent set $\mathcal{T}$. Then
$(K(\mathcal{T}) |K,v)$ is a valuation regular extension.
\end{lemma}

\begin{lemma}                                         \label{vrT}
Assume that $(F|K,v)$ is a valuation regular subextension of a valued
field extension $(\Omega|K,v)$. If\/ $\mathcal{T}$ is a standard
algebraically valuation independent set in $(\Omega|F,v)$, then also
$(F(\mathcal{T})|K(\mathcal{T}),v)$ is a valuation regular extension.
\end{lemma}
\begin{proof}
We write $\mathcal{T}=\{x_i \,,\,y_j\mid i\in I\,,\,j\in J\}$ where
the values $vx_i\,$, $i\in I$, are rationally independent over
$vF$, and the residues $\ovl{y}_i\,$, $j\in J$, are algebraically
independent over $\ovl{F}$. Since $(F|K,v)$ is valuation
regular, Lemma~\ref{charvalreg} shows that $vF/vK$ is torsion
free and that $\ovl{F}|\ovl{K}$ is regular. The former implies that also
$vF\,\oplus\,\bigoplus_{i\in I}\Z vx_i$ is torsion free modulo
$vK\,\oplus\,\bigoplus_{i\in I}\Z vx_i\,$. The latter implies
that also the extension $\ovl{F}(\ovl{y}_j\mid j\in J)\,|\, \ovl{K}
(\ovl{y}_j\mid j\in J)$ is regular (this follows from the fact that the
elements $\ovl{y}_j$ are algebraically independent over $\ovl{F}$).
Again by Lemma~\ref{charvalreg},
$(F(\mathcal{T}) |K(\mathcal{T}),v)$ is a valuation regular extension.
\end{proof}

\pars
Here is the reason why we consider valuation disjoint extensions:
\begin{proposition}                               \label{valdis-dl}
Take an extension $(F|K,v)$ of henselian fields and a finite algebraic
extension $(L|K,v)$, valuation disjoint from $(F|K,v)$ (in
$(\tilde{F},v)$). Then
\begin{equation}                            \label{ineqdh}
\mbox{\rm d}(L|K,v)\>\geq\> \mbox{\rm d}(L.F|F,v)\;.
\end{equation}
\end{proposition}
\begin{proof}
Since $vL|vK$ is disjoint from $vF|vK$ and $\ovl{L}|\ovl{K}$ is linearly
disjoint from $\ovl{F}|\ovl{K}$, we find that
\begin{equation}                               \label{ineqvr}
\left\{\begin{array}{rcl}
(v(L.F):vF) &\geq& (vL+vF:vF) = (vL:vK)\\ {}
[\ovl{L.F}:\ovl{F}]&\geq& [\ovl{L}.\ovl{F}:\ovl{F}]=[\ovl{L}:\ovl{K}]
\end{array}\right.
\end{equation}
Now inequality (\ref{ineqdh}) follows from this together with
$[(L.F)^h:F^h] \leq [L^h:K^h]$.
\end{proof}

\begin{corollary}                           \label{defalgclo}
Take a valuation regular extension $(F|K,v)$, fix an extension of $v$
from $F$ to $\tilde{F}$ and assume that $(K,v)$ and $(\tilde{K}.F,v)$
are defectless fields. Then also $(F,v)$ is a defectless field.
\end{corollary}
\begin{proof}
From Theorem~\ref{dl-hdl} we know that $(K^h,v)$ and $((\tilde{K}.F)^h
,v)$ are defectless fields and that it suffices to prove that $(F^h,v)$
is a defectless field. We will prove that every finite subextension
$E|F^h$ of $(\tilde{K}.F)^h|F^h$ is defectless. Since
$((\tilde{K}.F)^h,v)$ is a defectless field, it will then follow from
Lemma~2.3 of [K6] that $(F^h,v)$ is a defectless field.

Since $(\tilde{K}.F)^h=\tilde{K}.F^h$ by Lemma~\ref{algehens}, there is
a finite extension $L|K^h$ such that $E\subseteq L.F^h$. Since $(K^h,v)$
is a defectless field, we know that $\mbox{\rm d}(L|K^h,v)=1$. Since
$(F|K,v)$ is assumed to be valuation regular, the same holds for
$(F^h|K^h,v)$ by Corollary~\ref{valreg-h}. Thus it follows from
Lemma~\ref{valdis-dl} that $\mbox{\rm d}(L.F^h|F^h,v)\leq
\mbox{\rm d}(L|K^h,v)=1$. Using Lemma~\ref{md} we conclude
that $\mbox{\rm d}(E|F^h,v)=1$.
\end{proof}

%
%
%ÄÄÄÄÄÄÄÄÄÄÄÄÄÄÄÄÄÄÄÄÄÄÄÄÄÄÄÄÄÄÄÄÄÄÄÄÄÄÄÄÄÄÄÄÄÄÄÄÄÄÄÄÄÄÄÄÄÄÄÄÄÄÄÄÄÄÄÄÄÄÄ
%
\subsection{Henselized function fields}     \label{secthff}
Theorem~\ref{dl-hdl} shows that a valued field is defectless if and
only if its henselization is. So instead of working with a valued
function field, we will rather analyze its henselization. Such
a henselization will be called a \bfind{henselized function field}.
Every finite extension of a henselized function field is again a
henselized function field. Further, a henselized function field
$(F|K,v)$ will be called \bfind{henselized rational} {\bf with
generator} $x$ if $F=K(x)^h$. We will say that $(F|K,v)$ is
\bfind{henselized inertially generated} {\bf with generator} $x$ if
$(F,v)$ is a finite unramified extension of a henselized rational
function field with generator $x$.

We note that $(F|K,v)$ is henselized inertially generated with generator
$x$ if and only if there exists some $y\in F$ such that $F=K(x)^h(y)$,
$vy=0$ and $\ovl{K(x)}(\ovl{y})|\ovl{K(x)}$ is a finite separable
extension of degree
\begin{equation}                            \label{odeg=deg}
[\ovl{K(x)}(\ovl{y}):\ovl{K(x)}]\>=\>[K(x,y):K(x)]
\>=\>[K(x)^h(y):K(x)^h]\;.
\end{equation}
Indeed, if the latter holds, then $F|K(x)^h$ is unramified.
Conversely, if $F|K(x)^h$ is unramified, then $\ovl{F}|\ovl{K(x)}$
is a finite separable extension, of degree $[F:K(x)^h]$, and we can
choose some $\zeta\in\ovl{F}$ such that $\ovl{F}=\ovl{K(x)}(\zeta)$.
Take a monic polynomial $f$ with coefficients in the valuation ring of
$K(x)$ and such that $\ovl{f}$ is the minimal polynomial of $\zeta$ over
$\ovl{K(x)}$. By Hensel's Lemma, $f$ has a root $y\in F$ such that
$\ovl{y}=\zeta$. Since $\deg (f)=\deg (\ovl{f})$, we obtain the first
equation in (\ref{odeg=deg}). Consequently,
\[ [\ovl{K(x)}(\ovl{y}):\ovl{K(x)}]=[K(x,y):K(x)]\geq
[K(x)^h(y):K(x)^h]\geq [\ovl{K(x)}(\ovl{y}):\ovl{K(x)}]\;.\]
Hence, equality holds everywhere, showing that $F=K(x)^h(y)$ and that
the second equation in (\ref{odeg=deg}) holds. Note that $vF=vK(x)$ by
Lemma~\ref{prelBour}.

\parm
Take a henselized inertially generated function field $(F|K,v)$ of
transcendence degree $1$ without transcendence defect. Then either $\rr
vF/vK=1$ and there is a \bfind{value-transcendental} element $x\in F$,
i.e., its value $vx$ is rationally independent over $vK$, or $\trdeg
\ovl{F}|\ovl{K}=1$ and there is a \bfind{residue-transcendental} element
$x\in F$, i.e., $vx=0$ and the residue $\ovl{x}$ is transcendental over
$\ovl{K}$. In both cases, $x$ is transcendental over $K$ by
Lemma~\ref{prelBour}, hence $F|K(x)$ is finite.

Assume in addition that the
element $x$ is a generator for the henselized inertially generated
function field $(F|K,v)$. Then in the first case, we call $x$ a
\bfind{value-transcendental generator}, and in the second case a
\bfind{residue-transcendental generator} for $(F|K,v)$. In both cases,
we also call $x$ a \bfind{valuation-transcendental generator}. In the
first (the ``value-transcendental'') case, $\ovl{F}|\ovl{K}$ is finite,
and in the second (the ``residue-transcen\-den\-tal'') case, $vF/vK$ is
a finite torsion group (cf.\ Corollary~\ref{fingentb}).

If in addition $K$ is algebraically closed, then by Lemma~\ref{aaa},
$\ovl{F}=\ovl{K}$ is algebraically closed and $vK$ is divisible in the
value-transcendental case, and $vF=vK$ is divisible and $\ovl{K}$
algebraically closed in the residue-transcendental case. We see
that a henselized rational function field of transcendence degree 1 with
a value-transcendental generator over an algebraically closed field does
not admit proper unramified extensions, hence:
\begin{lemma}
Every henselized inertially generated function field of
transcendence degree 1 with a value-transcendental generator over an
algebraically closed field is a henselized rational function field.
\end{lemma}

\begin{lemma}               \label{taar}
Take an algebraically closed field $K$ and a finite extension $(E|F,v)$
within $F^r$. If $(F|K,v)$ is a henselized rational function field with
value-trans\-cen\-den\-tal generator, then so is $(E|K,v)$.
If $(F|K,v)$ is a henselized inertially generated function field with
residue-transcendental generator, then so is $(E|K,v)$, with the same
generator.
\end{lemma}
\begin{proof}
Case I): \ $(F|K,v)$ has a value-transcendental generator $x$, so
$F=K(x)^h$. Since $vF=vK(x)=vK\oplus \Z vx$,
$vK$ is divisible and $vE|vF$ is finite, we have that $vE=vK\oplus
\Z\frac{vx}{n}$ for some $n\geq 1$. Since $E\subset F^r$, $n$ is equal
to $[E:F]$ and prime to $\chara\ovl{K}$. As $\ovl{F}=\ovl{K}$ is
algebraically closed and $\ovl{E}|\ovl{F}$ is finite, we find that
$\ovl{E}=\ovl{F}$. Therefore, Hensel's Lemma can be used to find an
element $y$ in the henselian field $E$ such that $y^n = cx$ for some
$c\in F$ of value $0$. We have that $vy=\frac{vx}{n}$, $K(x)^h\subset
K(y)^h$, $\ovl{K(y)}=\ovl{K(x)}$ and
\begin{eqnarray*}
[E:K(x)^h] &\geq & [K(y)^h:K(x)^h]\>\geq\>(vK(y):vK(x))\cdot [\ovl{K(y)}
:\ovl{K(x)}]\\
& = & (vK(y):vK(x))\>\geq\>n\>=\>[E:K(x)^h]\;.
\end{eqnarray*}
Hence, equality holds everywhere, so $E=K(y)^h$. As $vy$ is
rationally independent over $vK$, $(E|K,v)$ is a henselized rational
function field with value-transcendental generator $y$.
\sn
Case II): \ $(F|K,v)$ has a residue-transcendental generator $x$. Then
$vF=vK$ is divisible, and since $vE/vF$ is finite, $vE=vF$. Hence
$(E|F,v)$ is unramified and $(E|K,v)$ is again a henselized inertially
generated function field with generator $x$.
\end{proof}

%
%
%ÄÄÄÄÄÄÄÄÄÄÄÄÄÄÄÄÄÄÄÄÄÄÄÄÄÄÄÄÄÄÄÄÄÄÄÄÄÄÄÄÄÄÄÄÄÄÄÄÄÄÄÄÄÄÄÄÄÄÄÄÄÄÄÄÄÄÄÄÄÄÄ
%
\section{Inseparably defectless fields}     \label{sectins}
We shall prove the ``inseparably
defectless'' version of the Generalized Stability Theorem:

\begin{proposition}                      \label{aiinsep}
Let $(F|K,v)$ be a valued function field without transcendence defect.
If $(K,v)$ is an inseparably defectless field, then so is $(F,v)$.
\end{proposition}
\begin{proof}
We choose a standard valuation transcendence basis $\mathcal{T}$ of
$(F|K,v)$ as in Corollary~\ref{fingentb}. Then $F|K(\mathcal{T})$ is a
finite extension. By Lemma~4.15 of [K6], every finite extension of an
inseparably defectless field is again an inseparably defectless field.
Hence it suffices to prove that $(K(\mathcal{T}),v)$ is an inseparably
defectless field.

Every finite purely inseparable extension $L$ of $K(\mathcal{T})$ is
contained in an extension $E = K'(\mathcal{T}^{1/p^m}) = K'(t^{1/p^m}\mid
t\in \mathcal{T})$ for a suitable $m\in \N$ and some finite purely
inseparable extension $K'$ of $K$. Since $K'|K$ is algebraic, we know
from Lemma~\ref{aaa} that $vK'/vK$ is a torsion group and
$\ovl{K'}|\ovl{K}$ is algebraic. Consequently, the values
$vx_i^{1/{p^m}}=\frac{vx_i}{p^m}\,$, $1\leq i\leq r$, are still
rationally independent over $vK'$, and the residues
$\ovl{y^{1/{p^m}}}_i=\ovl{y}_i^{1/{p^m}}$, $1\leq j\leq s$, are still
algebraically independent over $\ovl{K'}$. This proves that
$\mathcal{T}^{1/{p^m}}$ is a standard valuation transcendence basis of
$(E|K',v)$. Now Lemma~\ref{prelBour} shows that
\[vE= vK'\oplus \Z vx_1^{1/p^m}\oplus\ldots\oplus \Z vx_r ^{1/p^m}
=vK'\oplus\Z\frac{vx_1}{p^m}\oplus\ldots\oplus\Z\frac{vx_r}{p^m}\]
and that
\[\ovl{E}=\ovl{K'}\left(\ovl{y_1^{1/p^m}},\ldots,\ovl{y_s^{1/p^m}}
\right) = \ovl{K'}(\ovl{y_1}^{1/p^m},\ldots,\ovl{y_s}^{1/p^m})\;,\]
whence
\begin{eqnarray*}
[E:K(\mathcal{T})] & = & [K'(\mathcal{T}^{1/p^m}):K'(\mathcal{T})]\cdot
[K'(\mathcal{T}):K(\mathcal{T})]=p^{m(r+s)}\cdot [K':K]\\
& = & p^{mr}\cdot p^{ms}\cdot (vK':vK)\cdot [\ovl{K'}:\ovl{K}]\\
& = & p^{mr}\cdot (vK':vK)\cdot p^{ms}\cdot [\ovl{K'}:\ovl{K}]\\
& = & \left(vE:vK(\mathcal{T})\right)\cdot
\left[\ovl{E}:\ovl{K(\mathcal{T})} \right]
\end{eqnarray*}
since $(K'|K,v)$ is defectless by hypothesis.
This equation shows that $(E|K(\mathcal{T}),v)$ is defectless. Then by
Lemma~\ref{md}, also its subextension $(L|K(\mathcal{T}),v)$ is defectless.
\end{proof}

%
%ÄÄÄÄÄÄÄÄÄÄÄÄÄÄÄÄÄÄÄÄÄÄÄÄÄÄÄÄÄÄÄÄÄÄÄÄÄÄÄÄÄÄÄÄÄÄÄÄÄÄÄÄÄÄÄÄÄÄÄÄÄÄÄÄÄÄÄÄÄÄÄ
%
\section{Galois extensions of degree $p$}   \label{sectge}
In this section, we will consider the structure of Galois extensions
$E|F$ of degree $p$ of a henselized inertially generated function field
$(F|K,v)$ of rank $1$ and of transcendence degree $1$ with a
valuation-transcendental generator $x$. Throughout, we will assume
that
\begin{equation}                   \label{condK}
\left\{\begin{array}{l}
(K,v) \mbox{ is henselian and } p=\chara \ovl{K}>0\,,\\
\mbox{and $K$ is closed under $p$-th roots.}
\end{array}\right.
\end{equation}
If $\chara K=p$, then the latter condition is equivalent to $K$ being
perfect. If $\chara K=0$, then it implies that $K$ contains all $p$-th
roots of unity.

In Section~\ref{secthff} we have shown that if $K$ is algebraically
closed, then we may assume that $F$ has the following structure:
\sn
{\bf a) The value-transcendental case:} $(F|K,v)$ is henselized
rational, so we have that
\begin{equation}                   \label{vtc}
%\left.\begin{array}{l}
F = K(x)^h \mbox{\ \ is of rank 1 and
%characteristic } p > 0\\ x\mbox{\ \
$x$ is value-transcendental over $K$.}
%\end{array}\right\}
\end{equation}
In this case, $\ovl{F} = \ovl{K}$ and $vF = vK\oplus\Z vx$.

\sn
{\bf b) The residue-transcendental case:}
here we have
\begin{equation}                   \label{rtc}
\left\{\begin{array}{l}
F = K(x)^h(y) \mbox{ \ is of rank 1, where
$x$ is residue-transcendental over $K$,}\\
vy=0,\\
\mbox{$\ovl{F}=\ovl{K}(\ovl{x},\ovl{y})$ with $\ovl{K}$ relatively
algebraically closed in $\ovl{F}$, and}\\
\mbox{$\ovl{K}(\ovl{x},\ovl{y})|\ovl{K}(\ovl{x})$ separable with
$[\ovl{K}(\ovl{x},\ovl{y}):\ovl{K}(\ovl{x})]=[K(x,y):K(x)]\>$.}
\end{array}\right.
\end{equation}
In this case, $vF = vK$.

\pars
In what follows we will {\it not} assume that $K$ is algebraically
closed. Instead, we will assume (\ref{condK}) together with
(\ref{vtc}) or (\ref{rtc}), respectively.

\parb
If $\chara K=p$, then the Galois extension $E|F$ of degree $p$ is an
Artin-Schreier extension (cf.\ [L], Theorem~6.4). That is,
the extension is of the form
\begin{equation}              \label{AS}
E = F(\vartheta) \ \ \mbox{ where } \ \ a\>:=\>\vartheta^p-\vartheta
\in F\;.
\end{equation}
By the additivity of the Artin-Schreier polynomial $\wp(X)=X^p-X$, for
every $d\in F$ we have:
\begin{equation}              \label{t+c}
E = F(\vartheta-d),\ \ \ \wp(\vartheta-d) = \vartheta ^p-d^p
- \vartheta +d = a -d^p+d \;\in\; a+\wp(F)\;.
\end{equation}
This shows that we can replace $a$ by any other element of $a+\wp(F)$
without changing the Artin-Schreier extension.
Note that by Hensel's Lemma, $X^p-X-a$ has a root in the henselian
field $F$ whenever $va>0$. For the valuation ideal $\mathcal{M}_F$ of $F$,
we thus have
\begin{equation}                            \label{MFswp}
\mathcal{M}_F\>\subset\>\wp(F)\;.
\end{equation}

\parm
If $\chara K=0$ and $K$ contains all $p$-th roots of unity, then $E|F$
is a Kummer extension (cf.\ [L], Theorem~6.2). That is, the extension is
of the form
\begin{equation}                                 \label{Ke}
E = F(\vartheta) \ \ \mbox{ where } \ \ a\>:=\>\vartheta^p\in F\;.
\end{equation}
For every $d\in F^{\times}$ we have
\begin{equation}                                     \label{dpa}
E = F(\vartheta d),\ \ \ (\vartheta d)^p = a d^p\;\in\;
a (F^{\times})^p\;,
\end{equation}
showing that we can replace $a$ by any other element of $a
(F^{\times})^p$ without changing the extension $E|F$.

\pars
The main goal of this section is the proof of the following result:

\begin{proposition}                               \label{kl}
Assume that $(K,v)$ satisfies condition (\ref{condK}), Further, let
$(F,v)$ be of the form (\ref{vtc}) or (\ref{rtc}), and $(E|F,v)$ a
Galois extension of degree $p$. Then either $(E|F,v)$ is defectless or
there is a Galois extension $L|K$ of degree $p$ with non-trivial defect
such that $(L.E|L.F,v)$ is defectless.
\end{proposition}

\begin{corollary}                           \label{cordlvtc}
Let $(F|K,v)$ be a henselized inertially generated function field of
transcendence degree 1 and rank 1. If $(K,v)$ is a defectless field,
then every Galois extension $(E|F,v)$ of degree $p$ is defectless.
\end{corollary}

\begin{remark}
The normal forms we will derive will show that if we drop the condition
that $K$ be closed under $p$-th roots, then the assertion of
Proposition~\ref{kl} remains true if we replace $K$ by a finite
extension $K'$ (which depends on $F$), $F$ by $K'.F$ and $E$ by $K'.E$.
This extension $K'|K$ can be chosen purely inseparable if $\chara
K=p$, and to be generated by successive adjunction of $p$-th roots if
$\chara K=0$. Indeed, it suffices to choose $K'$ large enough to contain
the finitely many coefficients that appear in the normal forms.
\end{remark}

%
%ÄÄÄÄÄÄÄÄÄÄÄÄÄÄÄÄÄÄÄÄÄÄÄÄÄÄÄÄÄÄÄÄÄÄÄÄÄÄÄÄÄÄÄÄÄÄÄÄÄÄÄÄÄÄÄÄÄÄÄÄÄÄÄÄÄÄÄÄÄÄÄ
%
\subsection{The value-transcendental case}
We will first discuss the value-transcen\-den\-tal case of
Proposition~\ref{kl}.
We will consider the ring
\[R = K[x,x^{-1}]\]
which consists of all finite Laurent series
\begin{equation}                      \label{not}
\varphi(x) = \sum_{i\in I} c_i x^i\;,\;\;c_i \in K\,, \;\;\;
I\subset \Z \mbox{ finite.}
\end{equation}

\begin{lemma}                         \label{F=R+}
Let $F=K(x)^h$ where $(K(x),v)$ is of rank 1 and $x$ is
value-trans\-cen\-dental over $K$. Then $R$ is dense in $F$. If $\chara
K=p$, then this implies that
$F = R + \wp(F)$.
\end{lemma}
\begin{proof}
Since $(K(x),v)$ is of rank 1, we know from Lemma~\ref{rk1dense} that it
is dense in $K(x)^h$. Now it suffices to prove that
$R$ is dense in its quotient field $K(x)$. For this,
it is enough to show that for every $0 \not=\varphi(x) \in R$ and
$\alpha\in vK(x)$ there is some $\tilde{\varphi}(x)\in R$ such that
\[v\left(\frac{1}{\varphi(x)} - \tilde{\varphi}(x)\right)>\alpha\;.\]
Using the notation of (\ref{not}) we have that
\begin{equation}             \label{mink}
v\varphi(x) = \min_{i\in I} vc_i x^i = vc_k x^k
\end{equation}
for a unique $k\in I$ since $x$ is value-transcendental over $K$.
We write
\[\frac{1}{\varphi(x)} = \frac{c_k^{-1}x^{-k}}{1-\psi(x)}\;\;\;
\mbox{ with }\;\;\; \psi(x) = 1-c_k^{-1}x^{-k}\varphi(x)\in R\]
Since $v\psi(x)>0$ and $vK$ is archimedean by hypothesis, there is some
$n\in\N$ such that $(n+1)v\psi(x)>\alpha+vc_kx^k$. Then
by the geometric expansion, the element
\[\tilde{\varphi}(x):=c_k^{-1}x^{-k}\sum_{i=0}^{n}\psi(x)^i\in R\]
satisfies
\begin{eqnarray*}
v\left(\frac{1}{\varphi(x)}-\tilde{\varphi}(x)\right) & = &
v\left(\frac{1}{1-\psi(x)}\,-\,\sum_{i=0}^{n}\psi(x)^i\right)-vc_kx^k \\
 & = & (n+1)v\psi(x)-vc_kx^k\> >\> \alpha\;.
\end{eqnarray*}
We have now proved that $R$ is dense in $F$. This means that $F=R+
\mathcal{M}_F\,$. If $\chara K=p$, then by (\ref{MFswp}) it follows that
$F = R + \wp(F)$.
\end{proof}

\mn
$\bullet$ \ {\bf The equal characteristic case.}

\sn
We deduce the following normal form, which proves the
equal characteristic value-transcen\-den\-tal case of
Proposition~\ref{kl}:

\begin{proposition}                                  \label{ecvt}
Let $K$, $F$ and $E$ be as in the value-transcendental case of
Proposition~\ref{kl}, and assume that $\chara K=p$. Then either
\begin{equation}                            \label{1stcase}
E\>=\>F(\vartheta)\;\; \mbox{ where }\;\vartheta^p-\vartheta\in K\;,
\end{equation}
or
\[
E\>=\>F(\vartheta)\;\; \mbox{ where } \; \vartheta^p - \vartheta
\>=\> c_0+\sum_{i\in I} c_i x^i\,,\;\; c_i\in K
\]
with finite non-empty $I\subset p\Z\setminus \Z$ and such that
\[\forall i\in I:\; vc_i x^i< 0\]
and $\>\forall i\in I:\, vc_i x^i<vc_0\leq 0$ if $\,c_0\ne 0$.
\pars
If (\ref{1stcase}) does not hold, then $(vE:vF)=p$ and the extension
$(E|F,v)$ is defectless. If (\ref{1stcase}) holds then for
$L=K(\vartheta)$ we have that the extension $(L.E|L.F,v)$ is
trivial and hence defectless. In this case, $(E|F,v)$ is defectless
if $(L|K,v)$ is defectless.
\end{proposition}
\begin{proof}
We can assume that $E|F$ is of the form (\ref{AS}). By Lemma~\ref{F=R+}
and (\ref{t+c}), we can also assume that $a$ is a finite Laurent series
$\varphi(x)$ of the form (\ref{not}), and only containing summands of
value $\leq 0$. Again by (\ref{t+c}), we can replace a summand $c_{jp}
x^{jp}$ of $\varphi(x)$ by a summand $c_j' x^j$ with $c_j'=
c_{jp}^{1/p}\in K$ (as $K$ is assumed to be perfect). After a finite
repetition of this procedure we arrive at a finite Laurent series $c_0 +
\sum_{i\in I} c_i x^i$ with
$I\subset \Z\setminus p\Z$. In this procedure,
all summands remain of value $\leq 0$ and the original coefficient
$c_0\in K$ remains unchanged. Note that for all $i\in I$, $vc_ix^i<0$
because $i\ne 0$ and thus $ivx\notin vK$.

Assume now that (\ref{1stcase}) does not hold. Then $I\ne\emptyset$. We
may assume that $vc_0> vc_ix^i$ for all $i\in I$. Indeed, if this
is not the case then using (\ref{t+c}) $\nu$ times, we can replace $c_0$
by $c_0^{1/p^\nu}$. For large enough $\nu$, we obtain that $0\geq
vc_0^{1/p^\nu}=(vc_0)/p^\nu > vc_ix^i$ for all $i\in I$, because the
value group $vK(x)$ is archimedean by hypothesis.

Since the elements $x^i$, $i\in\Z$, are valuation independent over $K$,
there is $j\in I$ such that $v(c_0+\sum_{i\in I} c_i x^i) = vc_jx^j<0$.
Therefore, we must have that $v\vartheta<0$. It follows that
$v\vartheta^p=pv\vartheta<v\vartheta$ and consequently, $pv\vartheta=
v\vartheta^p= vc_jx^j$ by the ultrametric triangle law. As $j\notin
p\Z$, this value is not in $pvF$. Hence $(vE:vF) \geq p$. From the
fundamental inequality it then follows that $(vE:vF)=p$.

Finally, assume that (\ref{1stcase}) holds. Set $L=K(\vartheta)$.
Since $vF/vK$ is torsion free by assumption and $\ovl{F}=\ovl{K}$, we
know from Lemma~\ref{charvalreg} that $(F|K,v)$ is valuation regular.
Therefore, the algebraic extension $(L|K,v)$ is valuation disjoint from
$(F|K,v)$ and Lemma~\ref{valdis-dl} tells us that $\mbox{\rm d}(L|K,v)
\geq\mbox{\rm d}(L.F|F,v)=\mbox{\rm d}(E|F,v)$. Hence if $(L|K,v)$ is
defectless, then so is $(E|F,v)$.
\end{proof}

\mn
$\bullet$ \ {\bf The mixed characteristic case.}
\sn
The following normal form proves the mixed characteristic
value-transcendental case of Proposition~\ref{kl}:

\begin{proposition}                                   \label{mcvt}
Let $K$, $F$ and $E$ be as in the value-transcendental case of
Proposition~\ref{kl}, and assume that $\chara K=0$.
Then
\[E \>=\> F(\vartheta)\;\; \mbox{ where }\;\vartheta^p = x^m u\;,\]
with $m\in\{0,\ldots,p-1\}$ and $u\in K[x,x^{-1}]$ a $1$-unit of the
form
\[u\>=\>1 + \sum_{i\in I} c_i x^i\;,\hspace{0.8cm} c_i\in K,\]
with finite index set $I\subset\Z\setminus\{0\}$ and
\begin{equation}                            \label{condval}
\forall i\in I:\;\; 0< vc_i x^i \leq \frac{p}{p-1}vp\mbox{ \ and \ }
( i\in p\Z \>\Rightarrow\> vc_i x^i> vp)\;.
\end{equation}
If $vc_i x^i\geq\frac{1}{p-1}vp$ for all $i\in I$, then it may
in addition be assumed that $I\subset\Z\setminus p\Z$. If
$I\ne\emptyset$ then $v\sum_{i\in I} c_i x^i=\min_{i\in I} vc_i x^i$
is not divisible by $p$ in $vF$.
\pars
Further, $m\ne 0$ or $I\ne\emptyset$. In both cases, $\,(vE:vF)=p$.
\end{proposition}
\begin{proof}
We can assume that $E|F$ is of the form (\ref{Ke}). Since $vF=vK\oplus
\Z vx$ and $\ovl{F}=\ovl{K}$, we can write $a=c x^k u$ where $k\in \Z$,
$c\in K$ and $u\in F$ a $1$-unit. Using (\ref{dpa}) and our assumption
that $K$ is closed under $p$-th roots, we may replace $a$ by $x^m u$
where $m=(k\bmod p)$.
By Lemma~\ref{F=R+} and part a) of Corollary~\ref{1+y}, we may assume
that
\[u \>=\> 1 + \sum_{i\in I} c_i x^i \in R \mbox{ \ \ with \ \ } c_i\in
K\,,\; vc_i x^i\,\leq\,\frac{p}{p-1}\,vp \mbox{ \ and \ }
I\subset\Z \mbox{ \ finite.}\]
Since $u$ is a $1$-unit, we have $v\sum c_i x^i>0$, and since the
elements $x^i$, $i\in \Z$, are valuation independent, we find that
$vc_i x^i>0$ for each $i\in I$.

Since $K$ is closed under $p$-th roots, a monomial $c_i x^i$ is a
$p^{\nu}$-th power in $F$ if and only if $i\in p^{\nu}\Z$; indeed, if
$i\notin p^{\nu}\Z$, then $v(c_ix^i)^{1/p^\nu}=vc_i^{1/p^\nu}+
\frac{i}{p^\nu}vx\notin vK\oplus\Z vx = vF$. In a first step, we will
eliminate all $p$-th powers $\ne c_0$ of value $\leq vp$ in the above
sum. Assuming that there are any, let $\nu$ be the largest positive
integer such that $p^{\nu}$-th powers $\ne c_0$ appear. Suppose that
there are $n$ many, and write them as $z_1^{p^{\nu}}, \ldots,
z_n^{p^{\nu}}\,$. Denote by $y$ the sum of all other summands, so
that $u=1+y+\sum_{j=1}^{n}z_j^{p^{\nu}}$. Using (\ref{dpa}), we replace
$u$ by
\[u'\;:=\;\frac{u}{(1+\sum_{j=1}^{n}z_j)^{p^{\nu}}}\;=\;
\frac{1+\sum_{j=1}^{n}z_j^{p^{\nu}}}{(1+\sum_{j=1}^{n}z_j)^{p^{\nu}}}
\,+\,\frac{y}{(1+\sum_{j=1}^{n}z_j)^{p^{\nu}}} \;\>.\]
Using the geometrical series expansion as in the proof of Lemma~\ref{F=R+},
we find that the first quotient on the right hand side is equivalent to
$1$ modulo $p\mathcal{M}_F\,$.
The second quotient on the right hand side is equivalent modulo
$p\mathcal{M}_F$ to a finite sum of products of the summands in $y$ with
the $p^{\nu}$-th powers $z_1^{p^{\nu}}, \ldots,z_n^{p^{\nu}}\,$. Since
the summands $\ne c_0$ in $y$ aren't $p^{\nu}$-th powers, the products
aren't either. By the geometrical series expansion we also find that
there are only finitely many summands $c_i'x^i$ in $u'$ of value
$\leq\frac{p}{p-1}vp\,$. By part a) of Corollary~\ref{1+y} we can thus
replace $u'$ by the sum of these finitely many elements, which is a
$1$-unit. We will call it again $u$. Iterating this process until we
finish with $\nu=1$, we obtain a $1$-unit $u$ in which no $p$-th powers
$c_ix^i\ne c_0$ of value $\leq vp$ appear:
\begin{eqnarray*}
u\>=\> 1+\sum_{i\in I} c_i x^i & & \mbox{ with finite $I\subset\Z$,
$0<vc_ix^i\leq\frac{p}{p-1}vp\,$,}\\
& & \mbox{ and $c_i x^i\in p\mathcal{M}_F$ if } 0\ne i\in p\Z\,.
\end{eqnarray*}
We note that
\[\frac{u}{1+c_0}\>=\>\frac{1+c_0}{1+c_0}\,+\,\sum_{i\in I\setminus
\{0\}} \frac{c_i}{1+c_0} x^i \>=\> 1 \,+\,\sum_{i\in I\setminus \{0\}}
\frac{c_i}{1+c_0} x^i\]
where the values of the summands have not changed.
Since $K$ is closed under $p$-th roots, $1+c_0$ is a $p$-th power in $K$.
We may thus replace $u$ by $\frac{u}{1+c_0}$ and $I$ by $I\setminus\{0\}$
so that we may assume in addition that $I\subset\Z\setminus\{0\}$.

Suppose that all summands $c_i x^i$ have value $\geq\frac{1}{p-1}vp$.
Then part d) of Corollary~\ref{1+y} shows that we may replace every
monomial of the form $c_i x^i$ with $i\in p\Z$ by the monomial
$-p\tilde{c}_i x^{i/p}$ where $\tilde{c}_i\in K$ is a $p$-th root of
$c_i$. Note that $vc_i x^i\leq\frac{p}{p-1}vp$ implies that
$v(-p\tilde{c}_i x^{i/p})\leq\frac{p}{p-1}vp$. After an iterated
application we may then assume in addition that $I\cap p\Z = \emptyset$.

It remains to prove the last assertions of the lemma. If $m\ne 0$, then
$v\vartheta=\frac{m}{p}vx\notin vF$. Assume that $m=0$. Then $I\ne
\emptyset$ because the extension $E|F$ is assumed to be non-trivial.
Hence there is some $j\in I$ such that $v\sum_{i\in I} c_i x^i=vc_j
x^j$. We have that $j\notin p\Z\,$: if $vc_i x^i\geq\frac{1}{p-1}vp$ for
all $i\in I$ then this follows from $I\subset\Z\setminus p\Z$, and
otherwise it follows from the second assertion of (\ref{condval}). By
Lemma~\ref{exC} we know that $C\in K$. Performing the transformation
(\ref{transf}), we find that $\eta:=\frac{\vartheta-1}{C}$ is a root of
the polynomial (\ref{trpol}), with $b=\sum_{i\in I} c_i x^i$. Since
$vb=vc_j x^j < \frac{p}{p-1}vp\,$ (as $j\ne 0$, equality cannot
hold), we have that $v\frac{b}{C^p}<0$. The polynomial $g$ in
(\ref{trpol}) having coefficients $d_i$ of value $>0$, this implies that
$v\vartheta<0$. It follows that $v\vartheta^p<v\vartheta$ and
$v\vartheta^p<vd_i\vartheta^i$ for $2\leq i\leq p-1$. Consequently,
$pv\vartheta= v\frac{b}{C^p}$ by the ultrametric triangle law. This
yields:
\[v\eta\>=\>\frac{1}{p}\,vbC^{-p}\>=\>\frac{1}{p}\,vc_j-vC+\frac{j}{p}
\,vx \>\notin\> vK\oplus\Z vx\>=\>vF\;.\]
In all cases, we find as in the proof of Proposition~\ref{ecvt}
that $\,(vE:vF)=p$.
\end{proof}

%
%
%ÄÄÄÄÄÄÄÄÄÄÄÄÄÄÄÄÄÄÄÄÄÄÄÄÄÄÄÄÄÄÄÄÄÄÄÄÄÄÄÄÄÄÄÄÄÄÄÄÄÄÄÄÄÄÄÄÄÄÄÄÄÄÄÄÄÄÄÄÄÄÄ
%
\subsection{Frobenius-closed bases}
The ring $R=K[x,x^{-1}]$ that we used above has the
following crucial properties:\sn
-- $R$ contains $K$ and its quotient field $K(x)$ is dense in
$F=K(x)^h$,\sn
-- $R$ admits a valuation basis $\mathcal{B}=\{x^i\mid i\in I\}$ over $K$,
containing the element 1, such that the values $vx^i$, $i\in I$, form a
system of representatives of $vK(x)$ modulo $vK$,\sn
-- The basis $\mathcal{B}$ is \bfind{Frobenius-closed}, i.e., the $p$-th
power of every element in $\mathcal{B}$ lies again in $\mathcal{B}$.

We will need an analogue of such rings also in the
residue-transcendental case. But there, we will have to work with a
henselized inertially generated function field. The main goal of this
section is the proof of the following lemma:
\begin{lemma}                               \label{liftFcb}
Let $(F|K,v)$ be a henselized inertially generated function field of
rank $1$ and transcendence degree $1$ with a residue-transcendental
generator. Further, assume that $\ovl{K}$ is perfect of characteristic
$p>0$ and relatively algebraically closed in $\ovl{F}$, and that $K$ is
of arbitrary characteristic and closed under $p$-th roots. Then $F$
contains a subring $R$ which satisfies:
\begin{axiom}
\ax{(LFC1)} $R$ contains $K$ and its quotient field Quot($R$) is dense
in $F$,
\ax{(LFC2)} $R$ admits a valuation basis $\mathcal{B}=\{u_j\mid j\in J\}$
over $K$ of elements of value $0$ and containing the element 1, whose
residues $\ovl{u_j}$, $j\in J$, form a basis of $\ovl{F}|\ovl{K}$,
\ax{(LFC3)} $\mathcal{B}$ is Frobenius-closed.
\end{axiom}
\end{lemma}
\n
The valuation basis $\mathcal{B}$ of this lemma may be called a
\bfind{lifting of a Frobenius-closed basis} {\bf (LFC)} in view of
part a) of the following lemma. Part b) is the crucial property of
Frobenius-closed bases of function fields that we are going to
exploit.

\begin{lemma}                  \label{Frrs}
Assume that Properties (LFC2) and (LFC3) hold. Then:
\n
a) \ The basis $\ovl{\mathcal{B}}$ of $\ovl{F}|\ovl{K}$ consisting of all
$\ovl{u}_j$, $j\in J$, is also Frobenius-closed. If $\ovl{u}_m =
\ovl{u}_n^p$ then $u_m = u_n^p$.
\sn
b) \ If the sum
\[\sum_{i\in I} \ovl{c}_i \ovl{u}_i\,,\;\; \ovl{c}_i\in \ovl{K}\,,
\;I\subset J\mbox{\ finite}\]
is a $p$-th power, then for every $i\in I$ with $\ovl{c}_i\not= 0$, the
basis element $\ovl{u}_i$ is a $p$-th power of a basis element.
\sn
c) \ If
\[0\>\ne\>\vartheta^p-\vartheta\>=\>\sum_{i\in I} \ovl{c}_i\ovl{u}_i\,,
\;\; \ovl{c}_i\in \ovl{K}\,,\;I\subset J\mbox{\ finite}\]
then there is some $i\in I$ with $\ovl{c}_i\not= 0$ such that the
basis element $\ovl{u}_i$ is a $p$-th power of a basis element.
\end{lemma}
\begin{proof}
a): \ Since $\mathcal{B}$ is Frobenius-closed, every $u_j^p$ is an element
of $\mathcal{B}$. Hence $\ovl{u}_j^p = \ovl{u_j^p}\in \ovl{\mathcal{B}}$
which shows that $\ovl{\mathcal{B}}$ is Frobenius-closed. If $\ovl{u}_m =
\ovl{u}_n^p$ then $v(u_m - u_n^p) > 0 = vu_m$ which is only possible
if $u_m= u_n^p$ since $\mathcal{B}$ is assumed to be a valuation basis.

\sn
b): \ Assume that the sum is equal to
\[\left(\sum_{j\in J_0} \ovl{d}_j\ovl{u}_j\right)^p ,\;\;
\ovl{d}_j\in \ovl{K}\]
where $J_0\subset J$ is a finite index set. Then
\[\sum_{i\in I} \ovl{c}_i \ovl{u}_i \>=\> \sum_{j\in J_0}
\ovl{d}_j^p \ovl{u}_j^p\;.\]
By part a), the elements $\ovl{u}_j^p$ are also basis elements. Hence
every $\ovl{u}_i$ which appears on the left hand side (i.e., $\ovl{c}_i
\not= 0$) equals a $p$-th power $\ovl{u}_j^p$ appearing on the right
hand side.

\sn
c): \ Similar to b), hence left to the reader.
\end{proof}

We also deduce the following analogue of Lemma~\ref{F=R+}.

\begin{lemma}                  \label{ASL}
Let $F$ be henselian of rank 1. Then the properties (LFC1) and (LFC2)
imply that $R$ is dense in $F$. If $\chara K=p$, then $F=R+\wp(F)$.
\end{lemma}
\begin{proof}
In view of (LFC1), we have to show that $R$ is dense in its quotient
field, for which it suffices to show that for every $r\in R$ and each
$\alpha\in vF$ there is $r'\in R$ such that $v(\frac{1}{r}-r')>\alpha$.
Assume that there exists an element $s\in R^{\times}$ with $v(rs-1)>0$
and write
\[\frac{1}{r} = \frac{s}{1 - (1-rs)}\;.\]
Note that $1-rs\in R$ and proceed as in the proof of Lemma~\ref{F=R+}.
It remains to show the existence of $s$. The properties $K\subset R$
and (LFC2) imply that $vR=vF$ and $\ovl{R}=\ovl{F}$. Hence there is some
$s_1\in R$ such that $vrs_1=0$ and that the residue of the element
$rs_1\in R$ has an inverse in $\ovl{R}$, say $\ovl{s_2}$ for suitable
$s_2\in R$. Then the element $s=s_1s_2$ has the desired property since
$vrs_1s_2 = vrs_1= 0$ and $\ovl{rs_1s_2} = \ovl{1}$.

The last assertion follows as in the proof of Lemma~\ref{F=R+}.
\end{proof}

We will now prove Lemma~\ref{liftFcb}. Since $\ovl{K}$ is assumed
to be perfect and $\ovl{F}|\ovl{K}$ a function field of
transcendence degree $1$ with $\ovl{K}$ relatively algebraically
closed in $\ovl{F}$, Theorem~10 of [K5] shows that there exists a
Frobenius-closed basis of $\ovl{F}|\ovl{K}$. We have to lift this
basis to a Frobenius-closed basis of $F$ over $K$.
Recall that $F$ is of the form (\ref{rtc}). Let $f(x,y) = 0$ be
the irreducible equation for $x,y$ over $K$, normed such that $f$ has
coefficients of value $\geq 0$ with $\ovl{f}(X,Y) \not= 0$. By
(\ref{rtc}), the polynomials $f(X,Y)$ and $\ovl{f}(X,Y)$ have the same
degree in $Y$, and
\begin{equation}                                 \label{Abl}
\frac{\partial \ovl{f}}{\partial Y}(\ovl{x},\ovl{y}) \>\ne\> 0\;.
\end{equation}

%-----------------------------------------------------------------------

\bn
$\bullet$ {\bf The case of $\chara K=p$.}
\bn
Since $K$ is assumed to be henselian and perfect, it contains a field
of representatives for the residue field $\ovl{K}$ (we leave the easy
proof to the reader). We identify this field with $\ovl{K}$, so that
we can write
\begin{equation}                                     \label{rsub}
\ovl{K}\subset K\;\;\; \mbox{ with } \ovl{a}=a
\mbox{ for all } a\in\ovl{K}\;.
\end{equation}
This embedding can be extended to an embedding of $\ovl{F}$ in $F$ as
follows. By (\ref{rsub}), we can view the polynomial $\ovl{f}(x,Y)$ as a
polynomial over $K(x)\subset F$; from (\ref{Abl}) it follows by Hensel's
Lemma that this polynomial in $Y$ has exactly one zero $y'\in F$ with
$\ovl{y'} = \ovl{y}$. We have $K(x)^h(y')\subset K(x)^h(y)$. Again from
(\ref{Abl}) it follows that the polynomial $f(x,Y)$ has exactly one root
in $K(x)^h(y')$ whose residue is equal to $\ovl{y'} = \ovl{y}$. This
root must be $y$, hence $y\in K(x)^h(y')$ and we have shown that
\begin{equation}                     \label{=F}
K(x)^h(y') \>=\> K(x)^h(y) \>=\> F\;.
\end{equation}
The residue map induces on $\ovl{K}$ the identity and an isomorphism
\[\ovl{K}(x)\longrightarrow \ovl{K}(\ovl{x})\]
since both $x$ and $\ovl{x}$ are transcendental over $\ovl{K}$. It
leaves the coefficients of the irreducible polynomial $\ovl{f}(X,Y)$
fixed and sends the zero $(x,y')$ of $\ovl{f}(X,Y)$ to the zero
$(\ovl{x}, \ovl{y})$, hence it induces an isomorphism
\[\ovl{K}(x,y')\longrightarrow \ovl{K}(\ovl{x},\ovl{y}) = \ovl{F}\;.\]
By this isomorphism we identify
\[x=\ovl{x}\;,\;\;\;y' = \ovl{y}\]
such that
\begin{equation}                            \label{FqF}
\ovl{F}\subset F\;,\;\;\; F = (K.\ovl{F})^h\;,
\end{equation}
the latter being a consequence of (\ref{=F}).

\pars
By construction, $K$ and $\ovl{F}$ are linearly disjoint over
$\ovl{K}$. We form the subring generated by both fields in $F$:
\[R = K[\ovl{F}] \subset F\;.\]
By (\ref{FqF}), $F$ is the henselization of the quotient field of $R$.
Since the rank of $F$ is 1, the field Quot($R$) is dense in its
henselization by Lemma~\ref{rk1dense}, hence $R$ satisfies property
(LFC1).

Every $\ovl{K}$-basis of $\ovl{F}$ is at the same time a $K$-basis of
$R$. As the residue map induces the identity on $\ovl{F}$, every such
basis is a valuation basis of $R$ over $K$. Thus, $R$ satisfies (LFC2).
If we choose, as indicated above, a Frobenius-closed basis of
$\ovl{F}|\ovl{K}$ then $R$ together with this basis also satisfies
property (LFC3).

\parm
We summarize what we have proved:
\begin{lemma}                               \label{LFCec}
In the case of $\chara K=p$ there exists an embedding of the
residue field $\ovl{F}$ in $F$ respecting the residue map such that
$\ovl{K} = K \cap \ovl{F}$, that $K$ is linearly disjoint from $\ovl{F}$
over $\ovl{K}$ and that $F = (K.\ovl{F})^h$. The ring $R =
K[\ovl{F}] \subset F$ satisfies properties (LFC1), (LFC2) and (LFC3).
\end{lemma}

%-----------------------------------------------------------------------
\bn
$\bullet$ {\bf The case of $\chara K=0$.}
\bn
In this case, $K$ contains $\Q$ and its valuation induces the $p$-adic
valuation $v_p$ on $\Q$. Since $K$ is algebraically closed, it contains
a subfield $K_0$ such that $\ovl{K_0} = \ovl{K}$ and $vK_0=v_p\Q=\Z vp$;
this field can be constructed as follows. Take $\mathcal{T}$ to be a set
of preimages for a transcendence basis $\ovl{\mathcal{T}}$ of $\ovl{K}|
\Fp$; then $\mathcal{T}$ is an algebraically valuation independent set
and we have $\ovl{\Q(\mathcal{T})} = \Fp(\ovl{\mathcal{T}})$ and
$v\Q(\mathcal{T})=v\Q$ by Lemma~\ref{prelBour}. Now
$\ovl{K}\,|\,\ovl{\Q(\mathcal{T})}$ is an algebraic extension which can be
viewed as a (transfinite) tower of finite separable and finite purely
inseparable extensions. By induction we succesively lift all of these
extensions, preserving their degrees. The separable extensions are
lifted by Hensel's Lemma, using our assumption that $(K,v)$ is
henselian. The purely inseparable extensions can be lifted using our
assumption that $K$ is closed under $p$-th roots. We obtain a tower of
finite extensions, starting from the field $\Q(\mathcal{T})$; since all of
them have the same degree as the corresponding extensions of their
residue fields, the fundamental inequality shows that they preserve the
value group as $\Q(\mathcal{T})$, which is $\Z vp$. The union over this
tower is the desired field $K_0$.

Note that $\ovl{K_0(x)}=\ovl{K}(\ovl{x})$. So we may choose a polynomial
$g(X,Y)\in K_0[X,Y]$ with coefficients of value $\geq 0$ such that $g$
has the same degree in $Y$ as $f$, and
\[\ovl{g}(X,Y) \>=\> \ovl{f}(X,Y)\mbox{ \ \ and \ \ }
\ovl{g(x,Y)} \>=\> \ovl{f}(\ovl{x},Y)\;.\]
From (\ref{Abl}) it follows by Hensel's Lemma that $g(x,Y)$ has exactly
one zero $y'\in F$ with $\ovl{y'} = \ovl{y}$. We have $K(x)^h(y')\subset
K(x)^h(y)$. Again from (\ref{Abl}) it follows that the polynomial
$f(x,Y)$ has exactly one root in $K(x)^h(y')$ whose residue is equal to
$\ovl{y'} = \ovl{y}$. This root must be $y$, hence $y\in K(x)^h(y')$ and
we have shown that $K(x)^h(y') \>=\> K(x)^h(y) \>=\> F$.
\[K(x)^h(y') \>=\> K(x)^h(y) \>=\> F\;.\]
Hence we assume from now on that $y$ is algebraic\ over $K_0(x)$ with
\[[K_0(x,y):K_0(x)] =[\ovl{K_0}(\ovl{x},\ovl{y}):\ovl{K_0}(\ovl{x})]
= [\ovl{K}(\ovl{x},\ovl{y}):\ovl{K}(\ovl{x})] = [K(x,y):K(x)]\;.\]
In particular, this shows that the function field $F_0 := K_0(x,y)$
is linearly disjoint from $K$ over $K_0$. Moreover, we have
\[\ovl{F_0} = \ovl{K_0}(\ovl{x},\ovl{y}) = \ovl{F}\]
and, in view of Lemma~\ref{algehens},
\begin{equation}                     \label{=F'}
F = K(x)^h(y)=K(x,y)^h=(K.K_0(x,y))^h = (K.F_0)^h \;.
\end{equation}

Now we lift the Frobenius-closed basis $\ovl{\mathcal{B}}$ of $\ovl{F}|
\ovl{K}$ to $F_0\,$. First we observe that every basis element $\ovl{u}
\ne 1$ in $\ovl{\mathcal{B}}$ is an element of $\ovl{F}\setminus
\ovl{K}$
and hence transcendental over $\ovl{K}$, as $\ovl{K}$ is assumed to be
relatively algebraically closed in $\ovl{F}$. This implies that there
exists an integer $\nu = \nu(\ovl{u})$ such that $\ovl{u}\notin
\ovl{F}^{p^{\nu}}$. Consequently,
\[\ovl{\mathcal{B}} = \{1\}\cup\{\ovl{u}^{p^n}\mid n\in\N \mbox{ and }
\ovl{u} \in\ovl{\mathcal{B}}\setminus \ovl{F}^p\}\;.\]
For every $\ovl{u}\in \ovl{\mathcal{B}}\setminus \ovl{F}^p$ we choose an
element $u\in F_0$ with residue $\ovl{u}$. Let $\mathcal{B}'$ be the
collection of all these elements $u$. Then
\[\mathcal{B} = \{1\}\cup\{u^{p^n}\mid n\in \N \mbox{\ and\ }
u\in \mathcal{B}'\}\]
is a valuation independent set and a set of representatives
for $\mathcal{B}$. Let
\[R_0 = K_0[\mathcal{B}]\]
be the subring of $F_0$ generated over $K_0$ by the elements from
$\mathcal{B}$. Since
\[vR_0=v_p\Q=vF_0\mbox{\ \ and\ \ } \ovl{R_0} = \ovl{F} = \ovl{F_0}\]
and since the value group $v_p\Q$ is isomorphic to $\Z$, we conclude
that $R_0$ is dense in $F_0$.

\pars
We form the subring generated by $K$ and $R_0$ in $F$:
\[R := K[R_0] = K[\mathcal{B}]\subset F\;.\]
Since $R_0$ is dense in $F_0$ and $K$ is of rank 1 by assumption, $R$
is dense in the ring
\[R' := K[F_0] \subset F\;.\]
From (\ref{=F'}) it follows that $F$ is the henselization of the
quotient field of $R'$. Since the rank of $F$ is 1, the field
Quot($R'$) is dense in its henselization by Lemma~\ref{rk1dense}, and
the fact that $R$ is dense in $R'$ implies that Quot($R$) is dense in
Quot($R'$). Hence $R$ satisfies (LFC1).

By construction, $\mathcal{B}$ is a valuation basis of $R$ over $K$,
containing 1, closed under $p$-th powers and such that
$\ovl{\mathcal{B}}$
is a Frobenius-closed basis of $\ovl{F}\mid\ovl{K}$. Hence $R$
satisfies (LFC2) and (LFC3). We summarize what we have proved:

\begin{lemma}                               \label{LFCmc}
In the case of $\chara K=0$ there exists a subfield $K_0$ of $K$ and a
function field $F_0|K_0$ linearly disjoint from $K|K_0$ such
that $vF_0=vK_0$ is discrete, $\ovl{K_0}=\ovl{K}$ and $\ovl{F_0}=
\ovl{F}$. The field $F_0$ contains a valuation independent set
$\mathcal{B}$
including 1 and closed under $p$-th powers such that $\ovl{\mathcal{B}}$
is a Frobenius-closed basis of $\ovl{F}$ over $\ovl{K}$. The ring
$K_0[\mathcal{B}]$ is dense in $F_0$, and the ring $R=K[\mathcal{B}]$
satisfies properties (LFC1), (LFC2) and (LFC3).
\end{lemma}

\bn
%
%ÄÄÄÄÄÄÄÄÄÄÄÄÄÄÄÄÄÄÄÄÄÄÄÄÄÄÄÄÄÄÄÄÄÄÄÄÄÄÄÄÄÄÄÄÄÄÄÄÄÄÄÄÄÄÄÄÄÄÄÄÄÄÄÄÄÄÄÄÄÄÄ
%
\subsection{The residue-transcendental case}
\mbox{ }\mn
$\bullet$ \ {\bf The equal characteristic case.}
\mn
Using Lemma~\ref{LFCec}, we derive the following normal form, which
proves the equal characteristic residue-transcendental case of
Proposition~\ref{kl}:

\begin{proposition}                               \label{ecrt}
Let $K$, $F$ and $E$ be as in the residue-transcendental case of
Proposition~\ref{kl}, and assume that $\chara K=p$. Choose a ring $R$
with Frobenius-closed basis $\mathcal{B}$ in $F|K$ as in
Lemma~\ref{LFCec}. Then either
\begin{equation}                            \label{1stcasert}
E\>=\>F(\vartheta)\;\; \mbox{ where }\;\vartheta^p-\vartheta\in K\;,
\end{equation}
or
\[E = F(\vartheta)\;\; \mbox{ where }\;\vartheta^p - \vartheta =
\sum_{i\in I} c_i u_i \;,\hspace{0.8cm} 0\ne c_i\in K,\; u_i\in\mathcal{
B}\]
with $I$ a finite index set, and such that no element $u_i\ne 1$
is a $p$-th power in $\mathcal{B}$, and
\[\forall i\in I:\; vc_i u_i = vc_i\leq 0\;.\]

If $vc_ju_j<0$ for some $j\in I$ with $u_j\ne 1$, then $\ovl{E}|\ovl{F}$
is purely inseparable of degree $p$. If $vc_iu_i=0$ for all $i\in I$, then
$\ovl{E}|\ovl{F}$ is separable of degree $p$. In both cases, $(E|F,v)$
is defectless.

Suppose that $u_\ell=1$ for some $\ell\in I$ and $c_\ell u_\ell$ is the
only summand of value $<0$. Take $L$ to be the Galois extension of $K$
generated by a root of the polynomial $X^p-X-c_\ell\,$. If $I=\{\ell\}$
then (\ref{1stcasert}) holds and the extension $(L.E|L.F,v)$ is trivial.
Otherwise, $\ovl{L.E}|\ovl{L.F}$ is separable of degree $p$. In both
cases, $(L.E|L.F,v)$ is defectless and if $(E|F,v)$ has non-trivial
defect, then $(L|K,v)$ has non-trivial defect.
\end{proposition}
\begin{proof}
We can assume that $E|F$ is of the form (\ref{AS}). By
Lemma~\ref{ASL} and (\ref{t+c}), we can also assume that
\[a \>=\> \sum_{i\in I} c_iu_i\;\; \mbox{ where }\; c_i\in K,\;
u_i\in\mathcal{B}\]
with $I$ a finite index set and such that $vc_iu_i\leq 0$ for all
$i\in I$. Note that by (LFC2),
\begin{equation}                       \label{mcu}
va = \min_{i\in I} vc_iu_i \leq 0
\end{equation}

After a replacement procedure similar to the one used in the proof of
Proposition~\ref{ecvt}, we can assume that no element $u_i\ne 1$ is a
$p$-th power of another element in $\mathcal{B}$ (which by part b) of
Lemma~\ref{Frrs} implies that $\ovl{u_i}$ is not a $p$-th power in
$\ovl{\mathcal{B}}$). Note that if $u_i\not=1$, then $u_i\in F\setminus
K$, which shows that there exists an integer $\nu= \nu(u_i)$ such that
$u_i\notin F^{p^{\nu}}$, and so $u_i\notin \mathcal{B}^{p^{\mu}}$.

\pars
Assume that $vc_ju_j<0$ for some $j\in I$ with $u_j\ne 1$, and choose
$j$ such that
\[vc_ju_j\>=\>\min_{i\in I\,,\,u_i\ne 1} vc_iu_i<0\;.\]
If $u_\ell=1$ for some $\ell\in I$, then as we did for the element $c_0$
in the proof of Proposition~\ref{ecvt}, we can replace $c_\ell u_\ell =
c_\ell \in K$ by a suitable $p^\nu$-th root in $K$ whose value is larger
than $vc_ju_j\,$. We can thus assume that $vc_ju_j=\min_{i\in I}
vc_iu_i\,$. As in the proof of Proposition~\ref{ecvt} we find that
$v\vartheta=\frac{1}{p}vc_ju_j=\frac{1}{p}vc_j\,$. Set $d=c_j^{-1/p}\in
K$ and $\chi:=d\vartheta$. Then $v\chi=0$ and $vd>0$.
We obtain:
\[\chi^p-d^{p-1}\chi\>=\>d^p(\vartheta^p-\vartheta)
\>=\>c_j^{-1}\sum_{i\in I} c_i u_i\;.\]
Set $d_i=c_i/c_j\,$; then $vd_i\geq 0$ and $d_j = 1$.
Since $vd^{p-1}\chi>0$, we see that
\[\ovl{\chi}^{p}\>=\> \ovl{\sum_{i\in I} d_i u_i}\>=\>\sum_{i\in I}
\ovl{d}_i\ovl{u}_i \]
Since $\ovl{d}_j=1$
and $\ovl{u}_j\ne 1$ is not a $p$-th power in $\ovl{\mathcal{B}}$ by
construction, we can infer from part b) of Lemma~\ref{Frrs} that
$\ovl{\chi}\notin\ovl{F}$.
Thus, $[\ovl{E}:\ovl{F}]\geq p$. By the fundamental
inequality (\ref{fundineq}), equality holds. Since the extension is
generated by $\ovl{\vartheta}$, it is purely inseparable.

\pars
Assume that all summands $c_iu_i$ have value $0$.
Then also $\vartheta$ has value $0$, and
\[\ovl{\vartheta}^p-\ovl{\vartheta}\>=\>\sum_{i\in I} \ovl{c}_i
\ovl{u}_i\;.\]
If the polynomial $X^p-X-\sum\ovl{c}_i \ovl{u}_i$ were reducible, then
Hensel's Lemma would yield that $[E:F]<p$ in contradiction to our
assumption. Hence $[\ovl{E}:\ovl{F}] \geq p$, and as before it follows
that equality holds. As the extension is generated by $\ovl{\vartheta}$,
it is separable.

\pars
Assume that $u_\ell=1$ for some $\ell\in I$ and $c_\ell u_\ell$
is the only summand of value $<0$. Choose $\vartheta_0\in\tilde{K}$ such
that $\vartheta_0^p-\vartheta_0=c_\ell\,$, and set $L=K(\vartheta_0)$.
Then for $\vartheta_1:=\vartheta -\vartheta_0\in L.E$, we obtain from
(\ref{t+c}) that
\[L.E = L.F(\vartheta_1)\;\; \mbox{ where }\;\vartheta_1^p - \vartheta_1
=\sum_{i\in I\setminus \{\ell\}} c_i u_i \;.\]
If the polynomial $X^p-X-\sum_{\ell\ne i\in I}\ovl{c}_i \ovl{u}_i$ is
reducible, then it splits completely. But by part c) of
Lemma~\ref{Frrs}, this is only possible if the sum is empty, that is,
$I=\{\ell\}$. In this case, (\ref{1stcasert}) holds with
$\vartheta^p-\vartheta=c_\ell\,$.
If the polynomial is irreducible, then $[\ovl{L.E}:\ovl{L.F}]\geq p$,
and again, equality holds. As the extension $\ovl{L.E}|\ovl{L.F}$ is
generated by $\ovl{\vartheta}_1$, it is separable.

\pars
Finally, it remains to prove the very last assertion of the lemma. Since
$\ovl{K}$ is assumed to be perfect and relatively algebraically closed
in $\ovl{F}$, the extension $\ovl{F}|\ovl{K}$ is regular. Since also
$vF=vK$, we know from Lemma~\ref{charvalreg} that $(F|K,v)$ is valuation
regular. As in the proof of Proposition~\ref{ecvt} it follows that if
$(L.F|F,v)$ has non-trivial defect, then $(L|K,v)$ has non-trivial
defect. Suppose that $(E|F,v)$ has non-trivial defect. Since $E|F$ is a
subextension of $L.E|F$ it then follows from Lemma~\ref{md} that
$(L.E|F,v)$ has non-trivial defect. Since $(L.E|L.F,v)$ is defectless,
it follows again from Lemma~\ref{md} that $(L.F|F,v)$ and hence also
$(L|K,v)$ has non-trivial defect.
\end{proof}

%----------------------------------------------------------------------

\bn
$\bullet$\ \ {\bf The mixed characteristic case.}
\bn
Using Lemma~\ref{LFCmc}, we derive the following normal form, which
proves the mixed characteristic residue-transcendental case of
Proposition~\ref{kl}:

\begin{proposition}                               \label{mcrt}
Let $K$, $F$ and $E$ be as in the residue-transcendental case of
Proposition~\ref{kl}, and assume that $\chara K=0$. Choose a ring $R$
with Frobenius-closed basis $\mathcal{B}$ in $F|K$ as in
Lemma~\ref{LFCmc}. Then
\[E = F(\vartheta)\;\; \mbox{ where }\;\vartheta^p
\>=\> ru\]
such that $r\in R$ has value 0 and either $r=1$ or $\ovl{r}\notin
\ovl{F}^p$, and such that $u\in R$ is a $1$-unit of the form
\[u\>=\>1 + \sum_{i\in I} c_i u_i \;,\hspace{0.8cm} c_i\in K,\;
1\ne u_i\in \mathcal{B}\]
where $I$ is a finite index set and
\[\forall i\in I:\; 0 < vc_i u_i = vc_i\leq \frac{p}{p-1}vp
\mbox{ \ and \ } (u_i\in \mathcal{B}^p\>\Rightarrow\> vc_i > vp)\;.\]
If $vc_i\geq \frac{1}{p-1}vp$ for all $i\in I$, then it may
be assumed that no $u_i$ at all appearing in the sum is a $p$-th power
in $\mathcal{B}$.

\pars
In all cases, $[\ovl{E}:\ovl{F}]=p$, with the extension being separable
if and only if $r=1$ and $vc_i=\frac{p}{p-1}vp$ for all $i\in I$.
\end{proposition}
\begin{proof}
We can assume that $E|F$ is of the form (\ref{Ke}). Since $vF=vK$,
we can write $a=cb$ where $c\in K$ with $vc=va$. Using (\ref{dpa}) and
our assumption that $K$ is closed under $p$-th roots, we may replace $a$
by $b$. Now $vb=0$, so $\ovl{b}\ne 0$. If $\ovl{b}$ is a $p$-th power
in $\ovl{F}$, then we may choose $b_0\in F$ with $\ovl{b}_0=
\ovl{b}^{1/p}$ and write $b=b_0^p u$ with $u$ a $1$-unit. As before,
we may then replace $b$ by $u$. Now suppose that $\ovl{b}$ is not a
$p$-th power in $\ovl{F}$. By (LCF2) we can choose $r\in R$ such that
$\ovl{r}=\ovl{b}$ and write $b=ru$ with $u$ a $1$-unit. In summary, we
have shown that we can assume $\vartheta^p= ru$ with $r\in R$ of value 0
and either $r=1$ or $\ovl{r}\notin \ovl{F}^p$, and $u\in F$ a $1$-unit.

In view of Lemma~\ref{ASL} and part a) of Corollary~\ref{1+y}, we may
further assume that
\[u \>=\> 1+ \sum_{i\in I} c_i u_i\mbox{ \ \ with \ \ } c_i\in K,\;
u_i\in \mathcal{B}\,,\; vc_i u_i\,\leq\,\frac{p}{p-1}\,vp
\mbox{ \ and \ } I\subset\Z \mbox{ \ finite.}\]
Since $u$ is a $1$-unit, we have $v\sum c_i u_i>0$, and since the
elements $u_i$ form a valuation basis of $R$ over $K$ by (LFC2),
we find that $vc_i u_i>0$ for each $i\in I$.

All further properties stated for the summands $c_i u_i$ in the
proposition are now obtained by a replacdement procedure as in the proof
of Proposition~\ref{mcvt}.

\pars
It remains to prove the last assertions of the lemma. If $\ovl{r}\notin
\ovl{F}^p$ then $\ovl{\vartheta}=\ovl{r}^{1/p}\notin \ovl{F}$.

Now assume that $r=1$. By Lemma~\ref{exC} we know that $C\in K$.
Performing the transformation (\ref{transf}), we find that
\[\eta\>:=\>\frac{\vartheta-1}{C}\]
is a root of the polynomial (\ref{trpol}), where
$b=\sum_{i\in I} c_i u_i$. Suppose first that $vc_i u_i =
\frac{p}{p-1}vp$ for all $i\in I$. Then $vbC^{-p} = 0$, and the residue
polynomial $Z^p - Z - \ovl{bC^{-p}}$ does not admit a zero in $\ovl{F}$
since otherwise $\eta\in F$ by Hensel's Lemma and $E|F$ would be
trivial, contrary to our assumption that its degree is $p$.

Now suppose that $vc_i u_i < \frac{p}{p-1}vp$ for some $i\in I$. Then
for a suitable $j\in I$,
\begin{equation}                   \label{avs}
vb \>=\> \min_{i\in I}vc_i u_i \>=\> vc_j < \frac{p}{p-1}vp
\end{equation}
and $vbC^{-p}<0$. As in the proof of Proposition~\ref{mcvt} we find that
\[v\eta\>=\>\frac{1}{p}\,vbC^{-p}\>=\>\frac{1}{p}\,vc_j\,-\,vC\><\>0\;.\]
By our assumption on $K$ there is some $d\in K$ such that
$d^p=c_j^{-1}\,$. We have that $vdC\eta=0$ and $vdC>0$. Hence for
$\chi:=dC\eta$ we obtain:
\[\chi^p-(dC)^{p-1}\chi\>=\>(dC)^p(\eta^p-\eta)\>=\>(dC)^p \frac{b}{C^p}
\>=\>c_j^{-1}\sum_{i\in I} c_i u_i\;.\]
As in Proposition~\ref{ecrt} we see that $\ovl{\chi}\notin\ovl{F}$.

In all cases, we find that $[\ovl{E}:\ovl{F}]\geq p$ and hence,
$[\ovl{E}:\ovl{F}]=p$ by the fundamental inequality. It follows
that $\ovl{E}|\ovl{F}$ is a separable extension (generated by a root of
an Artin-Schreier polynomial) if $r=1$ and $vc_i=\frac{p}{p-1}vp$ for
all $i\in I$, and that it is a purely inseparable extension in the
remaining cases.
\end{proof}

%
%ÄÄÄÄÄÄÄÄÄÄÄÄÄÄÄÄÄÄÄÄÄÄÄÄÄÄÄÄÄÄÄÄÄÄÄÄÄÄÄÄÄÄÄÄÄÄÄÄÄÄÄÄÄÄÄÄÄÄÄÄÄÄÄÄÄÄÄÄÄÄÄ
%
\section{Proof of the Generalized Stability Theorem} \label{sectproof}
We prove Theorem~\ref{ai} by a stepwise reduction to the analysis of
Galois extensions of degree $p$ of certain valued function
fields without transcendence defect, followed by an application of the
results of Section~\ref{sectge}. Since every valued field $(K,v)$ of
residue characteristic $0$ is a defectless field, we will assume that
$\chara Kv=p>0$.
\sn
$\bullet$\ \ {\bf Reduction to transcendence degree 1.}
\begin{lemma}
To prove Theorem~\ref{ai}, it suffices to prove
\nn
{\bf (R1)}\ \ Every valued function field of transcendence
degree 1 without transcendence defect over a defectless ground field is
a defectless field.
\end{lemma}
\begin{proof}
By induction on the transcendence degree of the function field. The case
of transcendence degree 1 is covered by {\bf (R1)}.
Assume that $(F|K,v)$ is a valued function field of
transcendence degree $>1$ without transcendence defect. Choose any
subfunction field $F_0|K$ in $F|K$ such that $0<\trdeg F_0|K<\trdeg
F|K$. By the additivity of transcendence degree and rational rank, both
$(F_0|K,v)$ and $(F|F_0,v)$ are valued function fields without
transcendence defect. Hence if $(K,v)$ is a defectless field, then by
induction hypothesis, $(F_0,v)$ and consequently also $(F,v)$ is a
defectless field.
\end{proof}

\mn
$\bullet$\ \ {\bf Reduction to algebraically closed ground fields.}
\begin{lemma}
To prove {\bf (R1)}, it suffices to prove
\nn
{\bf (R2)}\ \ Every valued function field of transcendence
degree 1 without transcendence defect over an algebraically closed field
is a defectless field.
\end{lemma}
\begin{proof}
Let $(F|K,v)$ satisfy the assumptions of {\bf (R1)}.
We pick a valuation-trans\-cen\-dental element $t\in F$. By
Lemma~\ref{K(T)valreg}, $(K(t)|K,v)$ is a valuation regular extension,
and so is $(\tilde{K}(t)|\tilde{K},v)$ under every extension of $v$ to
$\tilde{K}(t)$. Hence by {\bf (R2)}, $(\tilde{K}(t),v)$ is a defectless
field. From Corollary~\ref{defalgclo} we infer that the same holds for
$(K(t),v)$. Since $F|K(t)$ is finite, it follows from
Corollary~\ref{mdc2} that also $(F,v)$ is a defectless field.
\end{proof}

\mn
$\bullet$\ \ {\bf Reduction to finite rank.}
\begin{lemma}
To prove {\bf (R2)}, it suffices to prove\nn
{\bf (R3)}\ \ Every valued function field of transcendence
degree 1 without transcendence defect over an algebraically closed
ground field of finite rank is a defectless field.
\end{lemma}
\begin{proof}
Let $(F|K,v)$ satisfy the assumptions of {\bf (R2)}. We wish to show
that $(F,v)$ is a defectless field. As in the preceding proof, we choose
a valuation transcendental element $t$, and by Corollary~\ref{mdc2} we
only have to show that $(K(t),v)$ is a defectless field. By
Theorem~\ref{dl-hdl} we may as well show that $(K(t)^h,v)$ is a
defectless field.

Let $(K(t)^h(a_1,\ldots,a_n)|K(t)^h,v)$ be an arbitrary finite
extension; we have to show that it is defectless. There exists a
finitely generated extension $k_1$ of the prime field of $K$ such
that $a_1,\ldots,a_n$ are already algebraic over $k_1(t)$. We take $k$
to be the algebraic closure of $k_1$ inside $K$. Since $k_1$ is finitely
generated over its prime field, the rank of $(k_1,v)$ and thus also of
its algebraic closure $(k,v)$ must be finite by Lemma~\ref{fatd=fr}. By
construction, $(k(t)|k,v)$ is a valued function field of
transcendence degree 1 without transcendence defect over the
algebraically closed ground field $(k,v)$ of finite rank.
So {\bf (R3)} together with Theorem~\ref{dl-hdl} implies that
$(k(t)^h(a_1,\ldots,a_n)|k(t)^h,v)$ is a defectless extension.

Since $k$ is algebraically closed, Lemma~\ref{charvalreg} shows that
the extension
$(K|k,v)$ is valuation regular. By Lemma~\ref{vrT} and
Corollary~\ref{valreg-h}, also the extensions $(K(t)|k(t),v)$ and
$(K(t)^h|k(t)^h,v)$ are valuation regular. Hence the algebraic extension
$(k(t)^h(a_1,\ldots,a_n)|k(t)^h,v)$ is valuation disjoint from
$(K(t)^h|k(t)^h,v)$. Now we apply Lemma~\ref{valdis-dl} to obtain
that
\[\mbox{\rm d}(K(t)^h(a_1,\ldots,a_n)|K(t)^h,v)\>\leq\>
\mbox{\rm d}(k(t)^h(a_1,\ldots,a_n)|k(t)^h,v) \>=\> 1\;.\]
We have thus proved that $(K(t)^h,v)$ is a defectless field.
\end{proof}

\mn
$\bullet$\ \ {\bf Reduction to rank 1.}
\begin{lemma}
To prove {\bf (R3)}, it suffices to prove\nn
{\bf (R4)}\ \ Every valued function field of rank 1 and of
transcendence degree 1 without transcendence defect over an
algebraically closed ground field is a defectless field.
\end{lemma}
\begin{proof}
Let $(F,v)$ satisfy the assumptions of {\bf (R3)}. Then $(F,v)$ must
also have finite rank (cf.\ (\ref{wtdgeq}) and note that the rank of
$vF$ cannot exceed the rank of $vK$ plus $\mbox{\rm rr}\,vL/vK$). Let
$v=w_1\circ\ldots\circ w_n$ be the decomposition of $v$ into valuations
$w_i$ of rank 1. By Lemma~\ref{aaa}, every $Kw_1\circ\ldots\circ w_i$ is
an algebraically closed field. By a repeated application of
Lemma~\ref{tdcomp}, $(F|K,w_1)$ and all $(Fw_1\circ\ldots\circ w_i|
Kw_1\circ\ldots\circ w_i,w_{i+1})$ are valued function fields without
transcendence defect. Hence (R4) yields that every
$(Fw_1\circ\ldots\circ w_i,w_{i+1})$ is a defectless field. Now a
repeated application of Lemma~\ref{defcow} shows that $(F,v)$ itself is
a defectless field.
\end{proof}

\pars
To complete our proof, we show that (R4) is true. Let $(F|K,v)$ satisfy
the conditions of {\bf (R4)}. Then there is a valuation-transcendental
element $x\in F$ and $F|K(x)$ is finite. By Corollary~\ref{mdc2} it
suffices to prove (R4) under the additional assumption that $F|K$ is
rational with a valuation-transcendental generator. In view of
Theorem~\ref{dl-hdl}, we can replace $(F,v)$ by its henselization.
More generally, we will now prove our assertion under the assumption
that $(F|K,v)$ is a henselized inertially generated function field of
rank 1 and transcendence degree 1 with a valuation-transcendental
generator over the algebraically closed field $K$. Given an arbitrary
finite extension $(E|F,v)$, we have to show that it is defectless.

Since the ramification group is a $p$-group (cf.\ [En]), $F\sep|F^r$ is
a $p$-extension. It follows from the general theory of $p$-groups (cf.\
[H], Chapter III, \S 7, Satz 7.2 and the following remark) via Galois
correspondence that the maximal separable subextension of $E.F^r|F^r$ is
a finite tower of Galois extensions of degree $p$. Consequently,
$E.F^r|F^r$ is a finite tower of normal extensions of degree $p$, either
Galois or purely inseparable. Then there is already a finite
subextension $N|F$ of $F^r|F$ such that $E.N|N$ is such a tower.
Lemma~\ref{taar} shows that $N$, being a finite subextension
inside $F^r$, is again a henselized inertially generated function field
of transcendence degree 1 with a valuation-transcendental generator over
$K$. Also, it is again of rank 1. By Proposition~\ref{dlta} we have
$\mbox{\rm d}(E|F,v)=\mbox{\rm d} (E.N|N,v)$, hence it suffices to prove
that $E.N|N$ is defectless. Since this is trivial if $E.N=N$, we assume
that $E.N\ne N$. Then the first extension in the tower is defectless by
Corollary~\ref{cordlvtc} or Proposition~\ref{aiinsep}. This yields
$\mbox{\rm d}(E|F,v)= \mbox{\rm d} (E.N|N,v) <[E.N:N]\leq [E:F]$, that
is, $(E|F,v)$ cannot be immediate. We have proved:
\sn
{\it Henselized inertially generated function fields of rank 1 and of
transcendence degree 1 with a valuation-transcendental generator over
an algebraically closed ground field do not admit proper immediate
algebraic extensions.}

\parm
We use this fact to prove:
\begin{lemma}
Every henselized function field $(F,v)$ of rank 1 and of
transcendence degree 1 without transcendence defect over an
algebraically closed ground field $K$ is a henselized inertially
generated function field with a valuation-transcendental generator.
\end{lemma}
\begin{proof}
Case I) \ Suppose $F$ contains a value-transcendental element $x$. Since
$\ovl{K(x)}=\ovl{K}$ is algebraically closed and $\ovl{F}|\ovl{K(x)}$ is
finite, we have $\ovl{F}=\ovl{K(x)}$. Further, $vF$ is a finite
extension of $vK(x)=vK\oplus \Z vx$. Since $vK$ is divisible, we have
that $vF=vK\oplus \Z\alpha$ for some $\alpha\in vF$. Choose $x'\in F$
such that $vx'=\alpha$. Then $vF= vK\oplus \Z vx'=vK(x')$ by
Lemma~\ref{prelBour}, and $\ovl{F} = \ovl{K} = \ovl{K(x')}$. Now the
henselian field $F$ contains the henselization $K(x')^h$, and we have
just shown that $F|K(x')^h$ is an immediate extension. But the
henselized inertially generated function field $K(x')^h$ of rank 1 and
of transcendence degree 1 with value-transcendental generator $x'$ over
$K$ does not admit any proper immediate algebraic extension. This yields
$F= K(x')^h$.

\sn
Case II) \ Suppose $F$ contains a residue-transcendental element $x$.
Since $vK(x) = vK$ is divisible and $vF/vK(x)$ is finite, we have $vF=
vK(x)$. Since $\ovl{F}|\ovl{K(x)}$ is finite and $\ovl{K(x)}=\ovl{K}
(\ovl{x})$ is a function field of transcendence degree $1$
over $\ovl{K}$, the same holds for $\ovl{F}$. Since $\ovl{K}$ is
algebraically closed, there is a separating transcendence basis
$\{\xi\}$ of $\ovl{F}| \ovl{K}$. We choose $x'\in F$ such that
$\ovl{x'}=\xi$. The henselian field $F$ contains the henselization
$K(x')^h$. Since $\ovl{F}|\ovl{K}(\xi)$ is a finite separable extension,
$\ovl{F}=\ovl{K}(\xi,\eta)$ for some $\eta$ which is separable-algebraic
over $\ovl{K}(\xi)$. Using Hensel's Lemma, we lift $\eta$ to an element
$y'\in F$ such that the finite subextension $K(x')^h(y')|K(x')^h$ of
$F|K(x')^h$ has the same degree as $\ovl{F}| \ovl{K}(\xi)$ and such that
$\ovl{K(x')^h(y')}=\ovl{F}$. It follows that $K(x')^h(y')|K(x')^h$ is
unramified and that $(K(x')^h(y')|K,v)$ is a henselized inertially
generated function field. Since $vK(x')^h(y')=vF$, we find that
$F|K(x')^h(y')$ is immediate. But as proved above, the henselized
inertially generated function field $K(x')^h(y')^h$ of rank 1 and of
transcendence degree 1 with residue-transcendental generator over $K$
does not admit any proper immediate algebraic extension. This yields $F=
K(x')^h(y')$.
\end{proof}

\pars
The proof that $E.N|N$ is defectless now proceeds by induction on the
number of extensions appearing in the tower. Since we have assumed that
$E.N\ne N$, there is a normal subextension $E'|N$ of $E.N|N$ of degree
$p$. By Corollary~\ref{cordlvtc} or Proposition~\ref{aiinsep}, this
extension is defectless. From the preceding lemma we infer that $E'$ is
again a henselized inertially generated function field of rank 1 and
transcendence degree 1 with a valuation-transcendental generator over
$K$. Its rank is 1 since it is an algebraic extension of $N$. By
induction hypothesis, $(E|E',v)$ is also defectless since it has a
smaller degree than $E|N$. Hence by Lemma~\ref{md}, $(E|N,v)$ is
defectless. This completes the proof of the first assertion of
Theorem~\ref{ai}.

\pars
The ``inseparably defectless'' version of Theorem~\ref{ai} has already
been proved in Section~\ref{sectins} (see Proposition~\ref{aiinsep}). So
it remains to prove the ``separably defectless'' version. Assume that
$(K,v)$ is a separably defectless field, $(F|K,v)$ is a valued
function field without transcendence defect, and that $vK$ is cofinal in
$vF$. Then the completion $F^c$ of $(F,v)$ contains the completion $K^c$
of $(K,v)$. A valued field is separably defectless if and only if its
completion is defectless (cf.\ [K6]). Hence, $K^c$ is a defectless
field. We consider the subfield $F.K^c\subset F^c$ which is a
function field over $K^c$. Since $(K^c|K,v)$ and $(F^c| F,v)$ are
immediate extensions and $\trdeg K^c.F|K^c\leq\trdeg F|K$,
$(K^c.F|K^c,v)$ is again a valued function field without
transcendence defect. By the ``defectless'' version of Theorem~\ref{ai}
it follows that $(K^c.F,v)$ is a defectless field. Hence also its
completion $(F^c,v)$ is defectless. By the above cited theorem it
follows that $(F,v)$ is defectless. This completes the
proof of the first assertion of Theorem~\ref{ai}.            \QED

\mn
$\bullet$ \ {\bf Proof of Corollary~\ref{aife}.} \
By Theorem~\ref{ai}, $(\tilde{K}.F,v)$ is a defectless field for every
extension of the valuation $v$. Therefore, $(\tilde{K}.F,v)$ is
defectless in $\tilde{K}.E$. Let $v_1,\ldots,v_{\rm g}$ be all
extensions of $v$ from $\tilde{K}.F$ to $\tilde{K}.E$. Then there
exists a finite extension $L_0|K$ such that:
\sn
1) \ $[L_0.E:L_0.F]=[\tilde{K}.E:\tilde{K}.F]$,
\n
2) \ the restrictions of $v_1,\ldots,v_{\rm g}$ to $L_0.E$ are all
distinct,
\n
3) \ $(v_i L_0.E:v_i L_0.F)\geq (v_i\tilde{K}.E:v_i\tilde{K}.F)$ and
$[L_0.E v_i:L_0.F v_i]\geq [\tilde{K}.E v_i:\tilde{K}.F v_i]$ for $1\leq
i\leq {\rm g}$. (In order to obtain the first inequality, choose
representatives of the cosets of $v_i\tilde{K}.E/
v_i\tilde{K}.F$ and choose $L_0$ so large that they lie in
$v_i L_0.E\,$. For the second inequality, choose a basis of
$\tilde{K}.E v_i|\tilde{K}.F v_i$ and choose $L_0$ so large that
it is contained in $L_0.E v_i\,$.)

\sn
These conditions will remain true whenever $L_0$ is replaced by any
algebraic extension $L$ of $K$ which contains $L_0\,$. Since equality
holds in the fundamental inequality for the extension $\tilde{K}.E$ of
$(\tilde{K}.F,v)$, the same is then true for the extension $L.E$ of
$(L.F,v)$.

Now assume in addition that $(K,v)$ is henselian. Then by a result of
Pank (cf.\ [K--P--R]) there exists a field complement $W$ to its
absolute ramification field $K^r$, that is, $W|K$ is linearly disjoint
from $K^r|K$ (and hence purely wild), and $K^r.W= \tilde{K}$.
Since $\tilde{W}= \tilde{K}=K^r.W=W^r$, $(W,v)$ is a defectless field.
Therefore, we can replace $\tilde{K}$ by $W$ in the above and choose
$L_0|K$ as a subextension of $W|K$ so that it is purely wild.       \QED

%
%ÄÄÄÄÄÄÄÄÄÄÄÄÄÄÄÄÄÄÄÄÄÄÄÄÄÄÄÄÄÄÄÄÄÄÄÄÄÄÄÄÄÄÄÄÄÄÄÄÄÄÄÄÄÄÄÄÄÄÄÄÄÄÄÄÄÄÄÄÄÄÄ
%
\section{Proof of Corollary~\ref{elram}}
The valued function field $(\tilde{K}.F|\tilde{K},v)$
satisfies the assumptions of Theorem~\ref{hrwtd}, so we may choose a
standard valuation transcendence basis $\mathcal{T}$ of $(\tilde{K}.F|
\tilde{K},v)$ with $v\tilde{K}.F= v\tilde{K}\oplus\Z vx_1\oplus
\ldots\oplus \Z vx_r$ and $\ovl{y}_1,\ldots,\ovl{y}_s$ a
separating transcendence basis of $\ovl{\tilde{K}.F}|\ovl{\tilde{K}}$
such that $[(\tilde{K}.F)^h:\tilde{K}(\mathcal{T})^h]=[\ovl{\tilde{K}.F}:
\ovl{\tilde{K}(\mathcal{T})}]$. Pick $\zeta\in\ovl{\tilde{K}.F}$ such that
$\ovl{\tilde{K}.F}= \ovl{\tilde{K}(\mathcal{T})}(\zeta)$ and take $\ovl{f}$
to be its separable minimal polynomial over $\ovl{\tilde{K}(\mathcal{T})}$.
Then there exists a finite extension $L_0|K$ such that:
\sn
1) \ $\mathcal{T}\subseteq L_0.F$,
\n
2) \ $[(L_0.F)^h:L_0(\mathcal{T})^h]=[(\tilde{K}.F)^h:\tilde{K}
(\mathcal{T})^h]$,
\n
3) \ $\ovl{L_0}(\ovl{y}_1,\ldots,\ovl{y}_s)$ contains all coefficients
of $\ovl{f}$.
\sn
These conditions will remain true whenever $L_0$ is replaced by any
algebraic extension $L$ of $K$ which contains $L_0\,$. For such $L$, we
have that $\mathcal{T}\subset L.F$ remains a valuation transcendence basis
of $(L.F|L(\mathcal{T}),v)$ since $L|K$ is algebraic. This yields that
$\ovl{L(\mathcal{T})}=\ovl{L}(\ovl{y}_1,\ldots,\ovl{y}_s)$ by
Lemma~\ref{prelBour}. Now we compute:
\begin{eqnarray*}
[(L.F)^h:L(\mathcal{T})^h] & \geq & [\ovl{L.F}:\ovl{L(\mathcal{T}}]
\>\geq\>[\ovl{L}(\ovl{y}_1,\ldots,\ovl{y}_s,\zeta):\ovl{L}(\ovl{y}_1,
\ldots,\ovl{y}_s)]\>=\> \deg \ovl{f}\\
 & = & [\ovl{\tilde{K}.F}:\ovl{\tilde{K}(\mathcal{T})}]
\>=\> [(\tilde{K}.F)^h:\tilde{K} (\mathcal{T})^h] \>=\>
[(L.F)^h:L(\mathcal{T})^h]\;.
\end{eqnarray*}
We see that equality must hold everywhere. In particular, the first
equality shows that the extension $((L.F)^h|L(\mathcal{T})^h,v)$ is
unramified.                                                  \QED

\newcommand{\lit}[1]{\bibitem #1{#1}}

\end{document}